\documentclass[12pt,reqno]{amsart}
\usepackage{amsmath, amsthm, amssymb}

\advance \topmargin by -\headheight
\advance \topmargin by -\headsep
     
\evensidemargin \oddsidemargin
\newtheorem{theorem}{Theorem}[section]
\newtheorem{lemma}[theorem]{Lemma}
\newtheorem{corollary}[theorem]{Corollary}
\newtheorem{conjecture}[theorem]{Conjecture}

\theoremstyle{definition}

\theoremstyle{remark}

\numberwithin{equation}{section}

\newcommand{\mmod}[1]{\,\,(\text{mod}\,\,#1)}

\def\bfa{{\mathbf a}}

 \def\bfF{{\mathbf F}}

\def\bfh{{\mathbf h}}

\def\bfm{{\mathbf m}}
\def\bfn{{\mathbf n}}

\def\bfv{{\mathbf v}}
\def\bfw{{\mathbf w}}
\def\bfx{{\mathbf x}}
\def\bfy{{\mathbf y}}
\def\bfz{{\mathbf z}}

\def\calB{{\mathcal B}}

\def\calP{{\mathcal P}}

\def\calR{{\mathcal R}}

\def\gtil{\tilde{g}}\def\ghat{\hat g}

\def\ftil{\widetilde{f}}

\def\dbC{{\mathbb C}}\def\dbD{{\mathbb D}}
\def\dbN{{\mathbb N}}
\def\dbR{{\mathbb R}}
\def\dbZ{{\mathbb Z}}\def\dbQ{{\mathbb Q}}

\def\gra{{\mathfrak a}}
\def\bfgra{{\boldsymbol \gra}}
\def\grb{{\mathfrak b}}\def\grB{{\mathfrak B}}
\def\bfgrb{{\boldsymbol \grb}}
\def\grc{{\mathfrak c}}

\def\grf{{\mathfrak f}}\def\grF{{\mathfrak F}}
\def\grg{{\mathfrak g}}
 \def\grG{{\mathfrak G}}
\def\grH{{\mathfrak H}}
\def\grJ{{\mathfrak J}}

\def\grm{{\mathfrak m}}\def\grM{{\mathfrak M}}
\def\grS{{\mathfrak S}}
\def\grB{{\mathfrak B}}

\def\alp{{\alpha}} \def\bfalp{{\boldsymbol \alpha}}
\def\bet{{\beta}}  \def\bfbet{{\boldsymbol \beta}}
\def\gam{{\gamma}} 
\def\del{{\delta}} \def\Del{{\Delta}}
\def\zet{{\zeta}} \def\bfzet{{\boldsymbol \zeta}} 
 
\def\tet{{\theta}} \def\bftet{{\boldsymbol \theta}} 

\def\lam{{\lambda}} \def\Lam{{\Lambda}} 

\def\bfxi{{\boldsymbol \xi}}

\def\Ups{{\Upsilon}} 

\def\ome{{\omega}} \def\Ome{{\Omega}}
\def\d{{\partial}}
\def\eps{\varepsilon}

\def\le{\leqslant} \def\ge{\geqslant}

\def\d{{\,{\rm d}}}

\def\llbracket{\lbrack\;\!\!\lbrack} \def\rrbracket{\rbrack\;\!\!\rbrack}

\begin{document}
\title[Discrete Fourier restriction]{Discrete Fourier restriction via efficient congruencing: 
basic principles}
\author[Trevor D. Wooley]{Trevor D. Wooley}
\address{School of Mathematics, University of Bristol, University Walk, Clifton, 
Bristol BS8 1TW, United Kingdom}
\email{matdw@bristol.ac.uk}
\date{}
\begin{abstract} We show that whenever $s>k(k+1)$, then for any complex sequence 
$(\gra_n)_{n\in \dbZ}$, one has 
$$\int_{[0,1)^k}\biggl| \sum_{|n|\le N}\gra_ne(\alp_1n+\ldots +\alp_kn^k)
\biggr|^{2s}\d\bfalp \ll N^{s-k(k+1)/2}\biggl( \sum_{|n|\le N}|\gra_n|^2\biggr)^s.$$
Bounds for the constant in the associated periodic Strichartz inequality from $L^{2s}$ to 
$l^2$ of the conjectured order of magnitude follow, and likewise for the constant in the 
discrete Fourier restriction problem from $l^2$ to $L^{s'}$, where $s'=2s/(2s-1)$. These 
bounds are obtained by generalising the efficient congruencing method from Vinogradov's 
mean value theorem to the present setting, introducing tools of wider application into the 
subject.
\end{abstract}
\maketitle

\section{Introduction} Our goal in this paper is to introduce a flexible new method for 
analysing a wide class of Fourier restriction problems, illustrating our approach in this first 
instance with a model example. In order to set the scene, consider a natural 
number $k$ and a function $g:(\dbR/\dbZ)^k\rightarrow \dbC$ having an associated 
Fourier series
\begin{equation}\label{1.1}
\gtil(\alp_1,\ldots ,\alp_k)=\sum_{\bfn\in \dbZ^k}\ghat(n_1,\ldots ,n_k)
e(n_1\alp_1+\ldots +n_k\alp_k).
\end{equation}
Here, we permit the Fourier coefficients $\ghat(\bfn)$ to be arbitrary complex numbers, 
and as usual write $e(z)=e^{2\pi i z}$. Beginning with work of Stein (see 
\cite{Ste1970}, and \cite{BCT2006} and \cite{Tao2004} for more recent broader context), 
there is by now an extensive body of research concerning the norms of 
operators restricting such Fourier series to integral points $\bfn$ lying on manifolds of 
various dimensions. Thus, for example, in work concerning the non-linear 
Schr\"odinger and KdV equations, Bourgain \cite{Bou1993a,Bou1993b} has considered 
the situation with $k=2$ and the restriction $\bfn=(n,n^l)$ to the Fourier series
$$\calR g=\sum_{n\in \dbZ}\ghat(n,n^l)e(n\alp_1+n^l\alp_2)\quad (l=2,3),$$
as well as higher dimensional analogues (see \cite{HL2013, HL2014} for recent work on 
these problems). Such results have also found recent application in 
additive combinatorics in work concerning the solutions of translation invariant equations 
with variables restricted to dense subsets of the integers (see \cite{Hen2014, Kei2014, 
Kei2015}).\par

We now focus on the example central to this paper, namely that in which the integral 
points $\bfn$ are restricted to lie on the curve $(n,n^2,\ldots ,n^k)$. When 
$(\gra_n)_{n\in \dbZ}$ is a sequence of complex numbers, write
\begin{equation}\label{1.2}
f_\bfgra (\bfalp ;N)=\sum_{|n|\le N}\gra_ne(\alp_1n+\alp_2n^2+\ldots +\alp_kn^k).
\end{equation}
Our goal is to obtain the periodic Strichartz inequality
\begin{equation}\label{1.3}
\| f_\bfgra(\bfalp;N)\|_{L^p}\le K_{p,N}\|\gra_n\|_{l^2}\quad (p\ge 2),
\end{equation}
with the sharpest attainable constant $K_{p,N}$, uniformly in $(\gra_n)$. By duality, this 
problem is related to the discrete restriction problem of obtaining the sharpest attainable 
constant $A_{p,N}$ for which
\begin{equation}\label{1.4}
\sum_{|n|\le N}|\ghat(n,n^2,\ldots ,n^k)|^2\le A_{p,N}\|g\|^2_{p'}\quad (p\ge 2),
\end{equation}
where $g:(\dbR/\dbZ)^k\rightarrow \dbC$, and $p'=p/(p-1)$. Indeed, one has 
$K_{p,N}\sim A_{p,N}^{1/2}$.\par

By adapting the efficient congruencing method introduced in work \cite{Woo2012} of the 
author associated with Vinogradov's mean value theorem, we obtain in \S\S6 and 7 the 
following conclusion.

\begin{theorem}\label{theorem1.1}
Suppose that $k\ge 2$ and $s\ge k(k+1)$. Then for any $\eps>0$, and any complex 
sequence $(\gra_n)_{n\in \dbZ}$, one has\footnote{We employ the convention that 
whenever $G:[0,1)^k\rightarrow \dbC$ is integrable, then $\oint G(\bfalp)\d\bfalp 
=\int_{[0,1)^k}G(\bfalp)\d\bfalp$. Moreover, constants implicit in Vinogradov's notation 
$\ll$ and $\gg $ may depend on $s$, $k$ and $\eps$.} 
\begin{equation}\label{1.5}
\oint |f_\bfgra(\bfalp;N)|^{2s}\d\bfalp \ll N^{s-k(k+1)/2+\eps}\biggl( \sum_{|n|\le N}
|\gra_n|^2\biggr)^s.
\end{equation}
Moreover, when $s>k(k+1)$, one may take $\eps=0$.
\end{theorem}

By way of comparison, we note that when $k=2$, Bourgain \cite{Bou1993a} has obtained 
the estimate
$$\oint |f_\bfgra(\bfalp;N)|^s\d\bfalp \ll 
N^\eps (1+N^{s-3})\biggl( \sum_{|n|\le N}|\gra_n|^2\biggr)^s,$$
in which the factor $N^\eps$ may be deleted whenever $s\ne 6$. Indeed, Bourgain shows 
that when $s=6$ this factor may be replaced by one of the shape $\exp(c\log N/\log \log 
N)$, and that it cannot be deleted. Also, forthcoming work of Kevin Hughes 
\cite{Hug2015} delivers a similar conclusion to that of Theorem \ref{theorem1.1}, though 
only for values of $s$ roughly twice as large as demanded by our new theorem.\par

Recall (\ref{1.3}) and (\ref{1.4}). In \S8 we show that the conclusion of Theorem 
\ref{theorem1.1} yields the following corollary providing bounds for $K_{p,N}$ and 
$A_{p,N}$.

\begin{corollary}\label{corollary1.2} When $k\ge 2$, $p\ge 2k(k+1)$ and $\eps>0$, one 
has
$$K_{p,N}\ll N^{(1-\tet)/2+\eps}\quad \text{and}\quad A_{p,N}\ll N^{1-\tet+\eps},$$
where we write $\tet=k(k+1)/p$. Moreover, provided that $p>2k(k+1)$, one may take 
$\eps=0$.\end{corollary}

Finally, a straightforward argument in \S8 conveys us from the estimates provided by 
Theorem \ref{theorem1.1} to the bounds recorded in the following corollary.

\begin{corollary}\label{corollary1.3}
Suppose that $t\ge 1$ and that $k_1,\ldots ,k_t$ are positive integers with $1\le 
k_1<k_2<\ldots <k_t=k$. Let $s\ge k(k+1)$, and write $K=k_1+\ldots +k_t$. Then for 
any $\eps>0$, and any complex sequence $(\gra_n)$, one has
$$\oint \biggl| \sum_{|n|\le N}\gra_ne(\alp_1n^{k_1}+\ldots +\alp_tn^{k_t})
\biggr|^{2s}\d\bfalp \ll N^{s-K+\eps}\biggl( \sum_{|n|\le N}|\gra_n|^2\biggr)^s.$$
Moreover, when $s>k(k+1)$, one may take $\eps=0$.
\end{corollary}

Bounds of the shape supplied by Theorem \ref{theorem1.1} are closely related to those 
available in Vinogradov's mean value theorem, which corresponds to the special case in 
which $(\gra_n)=(1)$. In the latter circumstances, one has the lower bound
\begin{equation}\label{1.6}
\oint \biggl| \sum_{|n|\le N}e(\alp_1n+\ldots +\alp_k n^k)\biggr|^{2s}\d\bfalp \gg 
N^s+N^{2s-k(k+1)/2}.
\end{equation}
It is therefore of interest to determine the least number $s_0=s_0(k)$ for which, for any 
$\eps>0$, one has the corresponding upper bound
$$\oint \biggl| \sum_{|n|\le N}e(\alp_1n+\ldots +\alp_k n^k)\biggr|^{2s}\d\bfalp \ll 
N^{2s-k(k+1)/2+\eps}.$$
The classical work of Vinogradov \cite{Vin1947} and Hua \cite{Hua1949} shows that one 
may take $s_0(k)\le 3k^2(\log k+O(\log \log k))$ for large $k$ (see Vaughan 
\cite[Theorem 7.4]{Vau1997}, for example). Very recently, with the arrival of the efficient 
congruencing method, such conclusions have become available for all $k\ge 2$ with 
$s_0(k)\le k(k+1)$ (see Wooley \cite[Theorem 1.1]{Woo2012}). The very latest 
refinements of such work \cite[Theorem 1.2]{Woo2014c} show that one may even take 
$s_0(k)\le k(k-1)$ whenever $k\ge 3$.\par

Following the discussion of the last paragraph, we are equipped to describe some 
consequences of forthcoming work \cite{Hug2015} of Kevin Hughes. This delivers a 
conclusion of the shape (\ref{1.5}) provided that $s\ge 2s_0(k)$, showing also that one 
may take $\eps=0$ for $s>2s_0(k)$. One can interpret Hughes' method as bounding 
mean values of $f_\bfgra(\bfalp;N)$ in terms of corresponding mean values of classical 
Vinogradov type with half as many underlying variables. This work demands essentially 
twice as many variables as are required in our Theorem \ref{theorem1.1}. In this paper, 
we work directly with mean values of $f_\bfgra(\bfalp;N)$, avoiding any reference to 
the classical version of Vinogradov's mean value theorem, and avoiding the losses 
inherent in previous approaches.\par

When $s>k(k+1)$, in the special case of the sequence $(\gra_n)=(1)$, the lower bound 
(\ref{1.6}) confirms that the upper bound (\ref{1.5}) of Theorem \ref{theorem1.1} is 
sharp (in which case one may take $\eps=0$). Indeed, by reference to this same special 
sequence, one may formulate the following conjecture.

\begin{conjecture}[Main Conjecture]\label{conjecture1.4} Suppose that $k\ge 1$. Then for 
any $\eps>0$, and any complex sequence $(\gra_n)$, one has
$$\oint |f_\bfgra (\bfalp;N)|^{2s}\d\bfalp \ll N^\eps \left(1+N^{s-k(k+1)/2}\right)
\biggl( \sum_{|n|\le N}|\gra_n|^2\biggr)^s.$$
\end{conjecture}

By orthogonality, the mean value on the left hand side here is a weighted count of the 
number of integral solutions of the Diophantine system
\begin{equation}\label{1.8}
\sum_{i=1}^sx_i^j=\sum_{i=1}^sy_i^j\quad (1\le j\le k),
\end{equation}
with $|x_i|,|y_i|\le N$. When $1\le s\le k$, it follows from Newton's formulae concerning 
roots of polynomials that $\{x_1,\ldots ,x_s\}=\{y_1,\ldots ,y_s\}$, and thus
$$\oint |f_\bfgra (\bfalp;N)|^{2s}\d\bfalp \ll \biggl( \sum_{|n|\le N}|\gra_n|^2\biggr)^s.$$
Thus, Conjecture \ref{conjecture1.4} holds when $1\le s\le k$ for essentially trivial reasons. 
With work, this line of reasoning can be extended to cover the case $s=k+1$. We note that 
Bourgain and Demeter \cite{BD2014} have confirmed this conjecture in the longer range 
$1\le s\le 2k-1$. By a substantial elaboration of the ideas of this paper, which we intend to 
pursue within a more general framework in a future memoir, we are able to extend this 
range substantially so as to confirm Conjecture \ref{conjecture1.4} in the interval 
$1\le s\le D(k)$, where $D(4)=8,\, D(5)=10,\, \ldots$, and 
$D(k)=\tfrac{1}{2}k(k+1)-(\tfrac{1}{3}+o(1))k$. Note that the confirmation of this 
conjecture for $1\le s\le k(k+1)/2$ would suffice to confirm it for all $s\ge 1$.\par

We have restricted ourselves in this paper to the relatively more straightforward proof of 
Theorem \ref{theorem1.1} so as to make the ideas underlying this generalisation of the 
efficient congruencing method more transparent. We hope, in this way, to permit other 
workers more easily to consider the use of such methods in applications farther afield.\par

Our basic strategy in proving Theorem \ref{theorem1.1} is to adapt to the present setting 
the efficient congruencing method introduced by the author in the context of Vinogradov's 
mean value theorem (see \cite{Woo2012}). Several complications must be surmounted in 
such a plan of attack. First, the presence of arbitrary complex coefficients $\gra_n$ implies 
that the strong translation-dilation invariance present in the setting of Vinogradov's mean 
value theorem is absent. However, we are able to normalise the exponential sums 
$f_\gra(\bfalp;N)$ by a factor $\left( \sum_{|n|\le N}|\gra_n|^2\right)^{-1/2}$ so as to 
achieve scale invariance, and subsequently consider mean values associated with extremal 
sequences. This initial preparation is discussed in \S2, and recovers a sufficient 
approximation to translation invariance that subsequent operations are not impeded.\par

We next impose a congruential condition, modulo an auxiliary prime number $\varpi$, on the 
summands underlying the mean value on the left hand side of (\ref{1.5}). By employing the 
weak translation-dilation invariance present in the underlying Diophantine system, as 
discussed in \S3, we find that a subset of the summands are subject to a strong congruence 
condition of the shape
$$\sum_{i=1}^kx_i^j\equiv \sum_{i=1}^ky_i^j\mmod{\varpi^j}\quad (1\le j\le k).$$
By incorporating a non-singularity condition {\it en passant} in \S4, one extracts in \S5 the 
condition $x_i\equiv y_i\mmod{\varpi^k}$ on the underlying variables. In this way, the 
original mean value is bounded above by a new mean value, both old and new containing 
complex weights, and the new one subject to powerful congruence constraints. By 
appropriate application of H\"older's inequality, this new mean value is bounded above by a 
further mean value encoding still stronger congruence constraints.\par

As in the efficient congruencing method for Vinogradov's mean value theorem, one now 
seeks to utilise this congruence concentration argument. In essence, if the original mean 
value is assumed to be significantly larger than conjectured, then one can show that a 
certain auxiliary mean value is significantly larger still by comparison to its expected size. By 
iterating, one ultimately arrives at a mean value shown to be so much larger than expected, 
that it exceeds even a trivial estimate for its magnitude. Thus one contradicts the initial 
assumption, and one is forced to conclude that the original mean value is of approximately 
the same magnitude as anticipated. This process is discussed in \S6. The additional 
complication associated with this approach concerns the complex weights $\gra_n$. For this 
reason, one must incorporate extra averaging by comparison with the earlier efficient 
congruencing approach.\par

One last matter deserves attention, namely that of the claim to the effect that one may 
take $\eps=0$ in the conclusion of Theorem \ref{theorem1.1} when $s>k(k+1)$. Such a 
conclusion first became available in the special case $k=2$ and $s>3$ in the work of 
Bourgain \cite{Bou1993a}. This approach has been generalised in forthcoming work of 
Kevin Hughes \cite{Hug2015}\footnote{Very recently, the author has been informed by Kevin 
Henriot that he has independently obtained such an $\eps$-removal lemma, and that this will 
appear in his forthcoming paper \cite{Hen2015}.}, so that whenever one has a bound of the 
shape
$$\oint |f_\bfgra (\bfalp;N)|^p\d\bfalp \ll N^{(p-k(k+1))/2+\eps}\biggl( \sum_{|n|\le N}
|\gra_n|^2\biggr)^{p/2},$$
valid for $p\ge p_0(k)$, then the same conclusion holds with $\eps=0$ provided that 
$p>p_0(k)$. Although we could apply Hughes' $\eps$-removal lemma, we have opted in \S7 
instead for a cheap treatment of slightly less generality in order that our account be 
self-contained. This approach is based on the Keil-Zhao device (see \cite{Kei2014} and 
\cite{Zha2014}), and we hope that it may be of independent interest.\par

It will be apparent to experts that the ideas introduced in this paper to surmount the 
difficulties associated with the complex weights $(\gra_n)$ are very flexible. Indeed, it 
seems that there is a metamathematical principle available that, when given a conventional 
unweighted mean value estimate established by a variant of the efficient congruencing 
method, delivers the corresponding estimate equipped with arbitrary complex weights 
through the methods of this paper. We discuss some of the immediate consequences of 
this principle in \S9.\par

The author is very grateful to Kevin Hughes for an inspiring seminar and subsequent 
conversations that provided motivation for this paper.

\section{The infrastructure for efficient congruencing} We initiate our discussion of the 
proof of Theorem \ref{theorem1.1} by introducing the components and basic notation 
required to assemble the weighted efficient congruencing iteration. Although analogous to 
that of our corresponding work \cite{Woo2012} concerning Vinogradov's mean value 
theorem, we are forced to deviate significantly from our earlier path.\par

Let $k$ be a fixed integer with $k\ge 2$, consider a complex sequence $(\gra_n)$, and 
recall the exponential sum $f_\bfgra(\bfalp;N)$ defined in (\ref{1.2}). Our first task is to 
replace the mean value central to Theorem \ref{theorem1.1} by a modification less 
intimately dependent on the sequence $(\gra_n)$. In this context, we remark that in the main 
thrust of our argument, we restrict attention to sequences other than the zero sequence 
$(\gra_n)=(0)$. Define $\rho(\bfgra;N)$ to be $1$ when $(\gra_n)=(0)$, and otherwise by taking
\begin{equation}\label{2.1}
\rho(\bfgra;N)=\biggl( \sum_{|n|\le N}|\gra_n|^2\biggr)^{1/2}.
\end{equation}
We define the normalised exponential sum
\begin{equation}\label{2.2}
\ftil_\bfgra(\bfalp;N)=\rho(\bfgra;N)^{-1}\sum_{|n|\le N}\gra_ne(\psi(n;\bfalp)),
\end{equation}
in which we put
\begin{equation}\label{2.2a}
\psi(n;\bfalp)=\alp_kn^k+\ldots +\alp_1n.
\end{equation}
Then, when $s>0$, we define the mean value
\begin{equation}\label{2.3}
U_{s,k}(N;\bfgra)=\oint |\ftil_\bfgra (\bfalp;N)|^{2s}\d\bfalp .
\end{equation}
A comparison of (\ref{1.2}) and (\ref{2.2}) reveals that
$$\oint |f_\bfgra(\bfalp;N)|^{2s}\d\bfalp \ll \rho(\bfgra;N)^{2s}U_{s,k}(N;\bfgra)=
U_{s,k}(N;\bfgra)\biggl( \sum_{|n|\le N}|\gra_n|^2\biggr)^s.$$
The proof of Theorem \ref{theorem1.1} will therefore be accomplished by establishing that, 
whenever $s\ge k(k+1)$, then for any $\eps>0$ and any complex sequence $(\gra_n)$, 
one has
$$U_{s,k}(N;\bfgra)\ll N^{s-k(k+1)/2+\eps},$$
and further that one may take $\eps=0$ when $s>k(k+1)$. Notice, in this context, that when 
$(\gra_n)=(0)$, then $U_{s,k}(N;\bfgra)=0$, so that we are entitled to ignore the zero 
sequence in subsequent discussion.\par

By applying Cauchy's inequality to (\ref{2.2}), one obtains the bound
$$|\ftil_\bfgra(\bfalp;N)|\le \rho(\bfgra;N)^{-1}\biggl( \sum_{|n|\le N}|\gra_n|^2
\biggr)^{1/2}(2N+1)^{1/2},$$
and thus it follows from (\ref{2.1}) that whenever $N\ge 1$, one has
\begin{equation}\label{2.X}
|\ftil_\bfgra(\bfalp;N)|\le 2N^{1/2},
\end{equation}
uniformly in $(\gra_n)$ and $\bfalp$. In particular, we infer from the definition (\ref{2.3}) 
that $U_{s,k}(N;\bfgra)\ll N^s$, uniformly in $(\gra_n)$. Further, by taking $(\gra_n)=(1)$ 
and applying orthogonality, one finds from (\ref{1.2}) that when $s\in \dbN$, the mean value 
$\oint |f_\bfgra(\bfalp;N)|^{2s}\d\bfalp $ counts the number of integral solutions of the 
system of equations (\ref{1.8}) with $|x_i|,|y_i|\le N$. The contribution of the diagonal 
solutions $x_i=y_i$ $(1\le i\le s)$ ensures in such circumstances that
$$\rho(\bfgra;N)^{-2s}\oint |f_\bfgra(\bfalp;N)|^{2s}\d\bfalp \gg (N^{1/2})^{-2s}N^s
=1.$$
We therefore deduce that when $(\gra_n)=(1)$, then one has $U_{s,k}(N;\bfgra)\gg 1$.\par

The definition of $U_{s,k}(N;\bfgra)$ ensures that it is non-negative for all $N$ and 
$(\gra_n)$. In addition, one sees from (\ref{2.2}) that $\ftil_\bfgra(\bfalp;N)$, and hence 
also $U_{s,k}(N;\bfgra)$, is scale invariant with respect to $(\gra_n)$, meaning that 
$U_{s,k}(N;\gam \bfgra)=U_{s,k}(N;\bfgra)$ for any $\gam>0$. We may therefore 
suppose without loss of generality that, for any fixed value of $N$, one has $|\gra_n|\le 1$ 
for $|n|\le N$. Write $\dbD$ for the unit disc $\{z\in \dbC: |z|\le 1\}$. Then it follows that for 
each fixed $N\ge 1$, one has
$$0\le \sup_{\substack{\gra_n\in \dbD\ (|n|\le N)\\ (\gra_n)\ne (0)}}\frac{\log 
U_{s,k}(N;\bfgra)}{\log N}\le s.$$
Thus, when $s\in \dbN$, we may define the quantity
\begin{equation}\label{2.6}
\lam_{s,k}=\limsup_{N\rightarrow \infty}\sup_{\substack{\gra_n\in \dbD\ (|n|\le N)\\ 
(\gra_n)\ne (0)}}\frac{\log U_{s,k}(N;\bfgra)}{\log N}.
\end{equation}

\par We make one further simplification before proceeding further, observing that there is no 
loss of generality in restricting the supremum in (\ref{2.6}) to be taken over complex 
sequences $(\gra_n)$ all of whose terms are real and positive. For given a complex 
sequence $(\gra_n)$, put $\grb_n=|\gra_n|$ for each $n$. Then it follows from (\ref{2.1}) 
that $\rho(\bfgrb;N)=\rho(\bfgra;N)$. Moreover, by orthogonality, the mean value 
$U_{s,k}(N;\bfgra)$ counts the number of integral solutions of the system of equations 
(\ref{1.8}) with $|x_i|,|y_i|\le N$, each solution $\bfx,\bfy$ being counted with weight
$$\rho(\bfgra;N)^{-2s}\gra_{x_1}\ldots \gra_{x_s}{\overline \gra}_{y_1}\ldots 
{\overline \gra}_{y_s}.$$
This weight has absolute value equal to
\begin{equation}\label{2.6z}
\rho(\bfgrb;N)^{-2s}\grb_{x_1}\ldots \grb_{x_s}{\overline \grb}_{y_1}\ldots 
{\overline \grb}_{y_s},
\end{equation}
and so by reversing the orthogonality argument, we find that
$$U_{s,k}(N;\bfgra)\le U_{s,k}(N;\bfgrb).$$
Moreover, should any sequence element $\grb_n$ be equal to $0$, then the weight 
(\ref{2.6z}) changes by a quantity lying in the interval $[0,\ome)$ when we substitute 
$\grb_n=\ome>0$. By repeating this process for each sequence element, and considering the 
limit as $\ome\rightarrow 0$, it is apparent that
$$\sup_{\substack{\gra_n\in \dbD\ (|n|\le N)\\ (\gra_n)\ne (0)}}
\frac{\log U_{s,k}(N;\bfgra)}{\log N}=\sup_{\grb_n\in (0,1]\ (|n|\le N)}
\frac{\log U_{s,k}(N;\bfgrb)}{\log N},$$
and so we are at liberty to restrict attention throughout the ensuing discussion to sequences 
of positive numbers. In particular, we may replace (\ref{2.6}) by the equivalent definition
$$\lam_{s,k}=\limsup_{N\rightarrow \infty}\sup_{\gra_n\in (0,1]\ (|n|\le N)}
\frac{\log U_{s,k}(N;\bfgra)}{\log N}.$$

\par In what follows, we fix, once and for all, the integer $k\ge 2$, and henceforth omit its 
explicit mention in our notation. Furthermore, we consider a natural number $u$ with 
$u\ge k$, and we put $s=uk$. Our method is iterative, with $N$ iterations, and we suppose 
this integer to be sufficiently large in terms of $s$ and $k$. We then put
\begin{equation}\label{2.7}
\tet=(Nk)^{-3N}\quad \text{and}\quad \del=(N^2s)^{-3N},
\end{equation}
so that, in particular, the parameter $\del$ is small compared to $\tet$. These quantities will 
shortly make an appearance that explains their role in the argument. Rather than discuss 
moments of order $2s$, it is more convenient to consider $U_{s+k}(X;\bfgra)$. The 
definition of $\lam=\lam_{s+k,k}$ in (\ref{2.6}) now ensures that there exists a sequence 
$(X_n)_{n=1}^\infty$ with $\underset{n\rightarrow \infty}{\lim}X_n=+\infty$ such that 
the following two statements hold whenever $m$ is sufficiently large. First, for some sequence 
$(\gra_n)$ of real numbers with $\gra_n\in (0,1]$, one has that
\begin{equation}\label{2.8}
U_{s+k}(X_m;\bfgra)>X_m^{\lam-\del}.
\end{equation}
Second, whenever $X_m^{1/2}\le Y\le X_m$, then for all non-zero complex sequences 
$(\gra'_n)$, one has
\begin{equation}\label{2.9}
U_{s+k}(Y;\bfgra')<Y^{\lam+\del}.
\end{equation}
We focus henceforth on a fixed element $X=X_m$ of the sequence $(X_n)$, which we may 
assume to be sufficiently large in terms of $s$, $k$, $N$ and $\del$. We also fix a real 
sequence $(\gra_n)$ with $\gra_n\in (0,1]$, satisfying (\ref{2.8}). It is convenient in 
the remainder of \S\S2--6 to abbreviate $U_{s+k}(X;\bfgra)$ to $U(X;\bfgra)$, or even to 
$U(X)$.\par

We next recall some standard notational conventions. The letter $\eps$ denotes a sufficiently 
small positive number. We think of the basic parameter as being $X$, a large real number 
depending at most on $\eps$ and the other ambient parameters as indicated. Whenever 
$\eps$ appears in a statement, we assert that the statement holds for each $\eps>0$. As 
usual, we write $\lfloor \psi\rfloor$ to denote the largest integer no larger than $\psi$, and 
$\lceil \psi\rceil$ to denote the least integer no smaller than $\psi$. We make sweeping use of 
vector notation. Thus, with $t$ implied from the environment at hand, we write 
$\bfz\equiv \bfw\mmod{\varpi}$ to denote that $z_i\equiv w_i\mmod{\varpi}$ $(1\le i\le t)$, 
or $\bfz\equiv \xi\mmod{\varpi}$ to denote that $z_i\equiv \xi\mmod{\varpi}$ $(1\le i\le t)$.

\par Congruences to prime power moduli lie at the heart of our argument. Put $M=X^\tet$, 
and note that $X^\del<M^{1/N}$. Let $\varpi$ be a fixed prime number with 
$M<\varpi\le 2M$ to be chosen in due course. That such a prime exists is a consequence of 
the Prime Number Theorem. When $c$ and $\xi$ are non-negative integers, we define 
$\rho_c(\xi)=\rho_c(\xi;\bfgra)$ by putting
\begin{equation}\label{2.10}
\rho_c(\xi)=\biggl( \sum_{\substack{|n|\le X\\ n\equiv \xi\mmod{\varpi^c}}}|\gra_n|^2
\biggr)^{1/2}.
\end{equation}
Note that, in terms of our earlier notation, one has
$$\rho_0(1)=\rho(\bfgra;X).$$
Moreover, one has the trivial relation
$$\rho_0(1)^2=\sum_{\xi=1}^{\varpi^c}\rho_c(\xi)^2.$$
In view of our assumption that $\gra_n\in (0,1]$ for each $n$, it is apparent that one has 
$\rho_c(\xi)>0$ whenever $1\le \xi<X$.\par

Recalling the notation (\ref{2.2a}), we introduce the normalised exponential sum
\begin{equation}\label{2.13}
\grf_c(\bfalp;\xi)=\rho_c(\xi)^{-1}\sum_{\substack{|n|\le X\\ n\equiv \xi\mmod{\varpi^c}}}
\gra_ne(\psi(n;\bfalp)).
\end{equation}
We find it necessary to consider {\it well-conditioned} $k$-tuples of integers belonging to 
distinct congruence classes modulo a suitable power of $\varpi$. Denote by $\Xi_c(\xi)$ the 
set of $k$-tuples $(\xi_1,\ldots ,\xi_k)$, with
$$1\le \xi_i\le \varpi^{c+1}\quad \text{and}\quad \xi_i\equiv \xi\mmod{\varpi^c}\quad 
(1\le i\le k),$$
and satisfying the property that $\xi_i\equiv \xi_j\mmod{\varpi^{c+1}}$ for no suffices $i$ 
and $j$ with $1\le i<j\le k$. We then put
\begin{equation}\label{2.14}
\grF_c(\bfalp;\xi)=\rho_c(\xi)^{-k}\sum_{\bfxi\in \Xi_c(\xi)}\prod_{i=1}^k
\rho_{c+1}(\xi_i)\grf_{c+1}(\bfalp;\xi_i).
\end{equation}

\par When $a$ and $b$ are non-negative integers, we define
\begin{align}
I_{a,b}(X;\xi,\eta)&=\oint |\grF_a(\bfalp;\xi)^2\grf_b(\bfalp;\eta)^{2s}|\d\bfalp ,
\label{2.15}\\
K_{a,b}(X;\xi,\eta)&=\oint |\grF_a(\bfalp;\xi)^2\grF_b(\bfalp;\eta)^{2u}|\d\bfalp ,
\label{2.16}
\end{align}
and then put
\begin{align}
I_{a,b}(X)&=\rho_0(1)^{-4}\sum_{\xi=1}^{\varpi^a}\sum_{\eta=1}^{\varpi^b}
\rho_a(\xi)^2\rho_b(\eta)^2I_{a,b}(X;\xi,\eta),\label{2.17}\\
K_{a,b}(X)&=\rho_0(1)^{-4}\sum_{\xi=1}^{\varpi^a}\sum_{\eta=1}^{\varpi^b}
\rho_a(\xi)^2\rho_b(\eta)^2K_{a,b}(X;\xi,\eta).\label{2.18}
\end{align}

\par By orthogonality, the mean value $I_{a,b}(X;\xi,\eta)$ counts the number of integral 
solutions of the system
\begin{equation}\label{2.w1}
\sum_{i=1}^k(x_i^j-y_i^j)=\sum_{l=1}^s(v_l^j-w_l^j)\quad (1\le j\le k),
\end{equation}
with
\begin{equation}\label{2.w1a}
|\bfx|,|\bfy|,|\bfv|,|\bfw|\le X,\quad \bfx,\bfy\in \Xi_a(\xi)\mmod{\varpi^{a+1}},
\end{equation}
and $\bfv\equiv \bfw\equiv \eta\mmod{\varpi^b}$, each solution being counted with weight
\begin{equation}\label{2.w2}
\rho_a(\xi)^{-2k}\rho_b(\eta)^{-2s}\Bigl( \prod_{i=1}^k\gra_{x_i}\gra_{y_i}\Bigr) 
\Bigl( \prod_{l=1}^s \gra_{v_l}\gra_{w_l}\Bigr) .
\end{equation}
Similarly, the mean value $K_{a,b}(X;\xi,\eta)$ counts the number of integral solutions of the 
system
\begin{equation}\label{2.w3}
\sum_{i=1}^k(x_i^j-y_i^j)=\sum_{l=1}^u\sum_{m=1}^k(v_{lm}^j-w_{lm}^j)\quad 
(1\le j\le k),
\end{equation}
subject to (\ref{2.w1a}) and $\bfv_l,\bfw_l\in \Xi_b(\eta)\mmod{\varpi^{b+1}}$ 
$(1\le l\le u)$, each solution being counted with weight
\begin{equation}\label{2.w4}
\rho_a(\xi)^{-2k}\rho_b(\eta)^{-2s}\Bigl( \prod_{i=1}^k\gra_{x_i}\gra_{y_i}\Bigr) 
\Bigl( \prod_{l=1}^u\prod_{m=1}^k \gra_{v_{lm}}\gra_{w_{lm}}\Bigr) .
\end{equation}
Given any one such solution to the system (\ref{2.w3}), an application of the Binomial 
Theorem shows that $\bfx-\eta$, $\bfy-\eta$, $\bfv-\eta$, $\bfw-\eta$ is also a solution. Since 
in any solution counted by $K_{a,b}(X;\xi,\eta)$, one has $\bfv\equiv \bfw\equiv 
\eta\mmod{\varpi^b}$, we deduce in particular that
\begin{equation}\label{2.w5}
\sum_{i=1}^k(x_i-\eta)^j\equiv \sum_{i=1}^k(y_i-\eta)^j\mmod{\varpi^{jb}}\quad 
(1\le j\le k).
\end{equation}

\par As in our previous work on efficient congruencing \cite{Woo2012}, our arguments are 
considerably simplified by making transparent the relationship between various mean values, 
on the one hand, and their anticipated magnitudes, on the other. For this reason we consider 
such mean values normalised by their anticipated orders of magnitude, as follows. We define
\begin{equation}\label{2.19}
\llbracket U(X)\rrbracket =\frac{U(X)}{X^{s+k-k(k+1)/2}},
\end{equation}
and, when $0\le a<b$, we define
\begin{align}
\llbracket I_{a,b}(X)\rrbracket &=\frac{I_{a,b}(X)}{(X/M^a)^{k-k(k+1)/2}(X/M^b)^s},
\label{2.20}\\
\llbracket K_{a,b}(X)\rrbracket &=\frac{K_{a,b}(X)}{(X/M^a)^{k-k(k+1)/2}(X/M^b)^s}.
\label{2.21}
\end{align}
In this notation, our earlier bounds (\ref{2.8}) and (\ref{2.9}) for $U(X)$ may be rewritten 
in the form
\begin{equation}\label{2.22}
\llbracket U(X)\rrbracket >X^{\Lam-\del}\quad \text{and}\quad \llbracket U(Y;\bfgra')
\rrbracket <Y^{\Lam+\del}\quad (Y\ge X^{1/2}),
\end{equation}
in which we write
\begin{equation}\label{2.23}
\Lam=\lam-(s+k)+k(k+1)/2.
\end{equation}

\par Our goal is to prove that $\Lam\le 0$ for $s\ge k^2$. This implies that whenever 
$\eps>0$ and $Z$ is sufficiently large in terms of $\eps$, $s$ and $k$, then for any 
non-zero complex sequence $(\grb_n)$, one has
$$U_{s+k}(Z;\bfgrb)\ll Z^{s+k-k(k+1)/2+\eps}.$$
Thus, whenever $s\ge k(k+1)$, $\eps>0$ and $Z$ is sufficiently large in terms of $\eps$, $s$ 
and $k$, it follows that for any complex sequence $(\grb_n)$, one has
$$\oint |f_\bfgrb(\bfalp;Z)|^{2s}\d\bfalp \ll Z^{s-k(k+1)/2+\eps}\biggl( \sum_{|n|\le Z}
|\grb_n|^2\biggr)^s,$$
which establishes the first claim of Theorem \ref{theorem1.1}. The second claim of Theorem 
\ref{theorem1.1} follows by application of the Keil-Zhao device, an argument we describe 
below in \S7.\par

Our strategy may now be outlined in vague terms. We first show that if $U(X)$ is not of the 
anticipated order of magnitude, then for some $\Lam>0$ one has $\llbracket K_{0,1}(X)
\rrbracket \gg X^{\Lam-\del}$, for a suitable large value of $X$. Next, for a sequence of 
integers $a_n$ and $b_n$ with $b_n$ roughly equal to $ka_n$, we show that $\llbracket 
K_{a_{n+1},b_{n+1}}(X)\rrbracket$ is always significantly larger than $\llbracket 
K_{a_n,b_n}(X)\rrbracket$. By iterating this process, we find that for suitably large $n$, the 
normalised mean value $\llbracket K_{a_n,b_n}(X)\rrbracket$ is so large that we contradict 
available upper bounds for its value. Thus, one is forced to conclude that $\Lam\le 0$, as 
desired.

\section{Some consequences of latent translation-dilation invariance} The presence of the 
coefficients $\gra_n$ prevents direct application of translation-dilation invariance in the mean 
values $U(X;\bfgra)$. However, our normalisation of the underlying exponential sums 
$\ftil_\bfgra(\bfalp;X)$, and definition via (\ref{2.6}) of the exponent $\lam$, ensures that 
much of the power of translation-dilation invariance can nonetheless be extracted. In this 
section we record the key consequences of this latent translation-dilation invariance for 
future reference.\par

We begin with an upper bound for a mean value analogous to $U(X;\bfgra)$.

\begin{lemma}\label{lemma3.1} Suppose that $c$ is a non-negative integer with 
$3c\tet\le 1$. Then
\begin{equation}\label{3.1}
\max_{1\le \xi\le \varpi^c}\oint |\grf_c(\bfalp;\xi)|^{2s+2k}\d\bfalp \ll 
(X/M^c)^{\lam+\del}.
\end{equation}
\end{lemma}

\begin{proof} Let $\xi$ be an integer with $1\le \xi\le \varpi^c$. From the definition 
(\ref{2.13}) of the exponential sum $\grf_c(\bfalp;\xi)$, one has
$$\grf_c(\bfalp;\xi)=\rho_c(\xi)^{-1}\sum_{-(X+\xi)/\varpi^c\le y\le (X-\xi)/\varpi^c}
\grb_y(\xi)e(\psi(\varpi^cy+\xi;\bfalp)),$$
in which $\psi(n;\bfalp)$ is given by (\ref{2.2a}), and
\begin{equation}\label{3.2}
\grb_y(\xi)=\gra_{\varpi^cy+\xi}.
\end{equation}
By orthogonality, one finds that the integral on the left hand side of (\ref{3.1}) counts the 
number of integral solutions of the system of equations
\begin{equation}\label{3.3}
\sum_{i=1}^{s+k}(\varpi^cy_i+\xi)^j=\sum_{i=1}^{s+k}(\varpi^cz_i+\xi)^j\quad 
(1\le j\le k),
\end{equation}
with $-(X+\xi)/\varpi^c\le \bfy,\bfz\le (X-\xi)/\varpi^c$, each solution being counted with 
weight
$$\rho_c(\xi)^{-2s-2k}\grb_{y_1}\ldots \grb_{y_{s+k}}\grb_{z_1}\ldots \grb_{z_{s+k}}.$$

\par An application of the Binomial Theorem shows that the pair $\bfy,\bfz$ satisfies 
(\ref{3.3}) if and only if it satisfies the system
$$\sum_{i=1}^{s+k}y_i^j=\sum_{i=1}^{s+k}z_i^j\quad (1\le j\le k).$$
Thus, recalling (\ref{2.2}), reversing track and accommodating the end-points of the 
summation, we find that
\begin{equation}\label{3.4}
\oint |\grf_c(\bfalp;\xi)|^{2s+2k}\d\bfalp =\oint |\ftil_{\bfgrb'}(\bfalp;(X+\xi)/\varpi^c)
|^{2s+2k}\d\bfalp ,
\end{equation}
where
$$\grb'_y=\begin{cases} \grb_y(\xi),&
\text{when $-(X+\xi)/\varpi^c\le y\le (X-\xi)/\varpi^c$,}\\
0,&\text{when $(X-\xi)/\varpi^c<y\le (X+\xi)/\varpi^c$.}\end{cases}$$
Here, we have made the trivial observation that, in view of the relation (\ref{3.2}), one has
\begin{align*}
\rho(\bfgrb';(X+\xi)/\varpi^c)^2&=\sum_{-(X+\xi)/\varpi^c\le y\le (X-\xi)/\varpi^c}
|\grb_y(\xi)|^2\\
&=\sum_{\substack{|n|\le X\\ n\equiv \xi\mmod{\varpi^c}}}|\gra_n|^2=\rho_c(\xi)^2.
\end{align*}

\par Observe next that the hypothesis $3c\tet\le 1$ ensures that $(X+\xi)/\varpi^c\ge 
X/M^c>X^{1/2}$. Thus, the upper bound (\ref{2.9}) supplies the estimate
\begin{align*}
\oint |\ftil_{\bfgrb'}(\bfalp;(X+\xi)/\varpi^c)|^{2s+2k}\d\bfalp &=U((X+\xi)/\varpi^c;\bfgrb')
\\
&<((X+\xi)/\varpi^c)^{\lam+\del}\ll (X/M^c)^{\lam+\del}.
\end{align*}
The desired conclusion now follows by substituting this bound into (\ref{3.4}).
\end{proof}

We record an additional estimate to demystify a bound of which we make use in our 
discussion of the congruencing step.

\begin{lemma}\label{lemma3.1a} Suppose that $c$ is a non-negative integer with 
$3c\tet\le 1$. Then
$$\max_{1\le \xi\le \varpi^c}\oint |\grF_c(\bfalp;\xi)|^{2u+2}\d\bfalp \ll 
(X/M^c)^{\lam+\del}.$$
\end{lemma}

\begin{proof} From the definitions (\ref{2.13}) and (\ref{2.14}) of $\grf_c(\bfalp;\xi)$ 
and $\grF_c(\bfalp;\xi)$, one has
\begin{equation}\label{3.w1}
\oint |\grF_c(\bfalp;\xi)|^{2u+2}\d\bfalp =\rho_c(\xi)^{-2s-2k}\oint \biggl| 
\sum_{\bfxi\in \Xi_c(\xi)}\prod_{i=1}^k\grg_{c+1}(\bfalp;\xi_i)\biggr|^{2u+2}\d\bfalp ,
\end{equation}
where we temporarily write
$$\grg_d(\bfalp;\zet)=\sum_{\substack{|n|\le X\\ n\equiv \zet\mmod{\varpi^d}}}
\gra_ne(\psi(n;\bfalp)).$$
We presume that the weights $\gra_n$ are positive, and hence the mean value on the right 
hand side of (\ref{3.w1}) counts the number of integral solutions of the system of equations
$$\sum_{i=1}^{s+k}(x_i^j-y_i^j)=0\quad (1\le j\le k),$$
with $|\bfx|,|\bfy|\le X$, $\bfx\equiv \bfy\equiv \xi\mmod{\varpi^c}$ and certain additional 
congruence conditions imposed by conditioning hypotheses, each solution being counted with 
the non-negative weight
$$\rho_c(\xi)^{-2s-2k}\gra_{x_1}\ldots \gra_{x_{s+k}}\gra_{y_1}\ldots \gra_{y_{s+k}}.
$$
Since these weights are non-negative, the omission of the additional congruence conditions 
cannot decrease the resulting estimate, and thus
\begin{align*}
\oint |\grF_c(\bfalp;\xi)|^{2u+2}\d\bfalp &\le \rho_c(\xi)^{-2s-2k}\oint 
|\grg_c(\bfalp;\xi)|^{2s+2k}\d\bfalp\\
&=\oint|\grf_c(\bfalp;\xi)|^{2s+2k}\d\bfalp .
\end{align*}
The conclusion of the lemma is now immediate from Lemma \ref{lemma3.1}.
\end{proof}

A variant of Lemma \ref{lemma3.1} proves useful both in this section and elsewhere.

\begin{lemma}\label{lemma3.2} Suppose that $c$ and $d$ are non-negative integers 
satisfying the condition $\max\{3c\tet,3d\tet\}\le 1$. Then
$$\max_{1\le \xi\le \varpi^c}\max_{1\le \eta\le \varpi^d}\oint |\grf_c(\bfalp;\xi)^{2k}
\grf_d(\bfalp;\eta)^{2s}|\d\bfalp \ll \left( (X/M^c)^k(X/M^d)^s\right)^{(\lam+\del)/(s+k)}.
$$
\end{lemma}

\begin{proof} Let $\xi$ and $\eta$ be integers with $1\le \xi\le \varpi^c$ and 
$1\le \eta \le \varpi^d$. An application of H\"older's inequality reveals that
\begin{align*}
\oint |\grf_c(\bfalp;\xi)^{2k}&\grf_d(\bfalp;\eta)^{2s}|\d\bfalp \\
&\le \biggl( \oint |\grf_c(\bfalp;\xi)|^{2s+2k}\d\bfalp \biggr)^{k/(s+k)}\biggl( \oint 
|\grf_d(\bfalp;\eta)|^{2s+2k}\d\bfalp \biggr)^{s/(s+k)}.
\end{align*}
The conclusion of the lemma follows from this bound by inserting the estimate supplied by 
Lemma \ref{lemma3.1} to estimate the mean values occurring on the right hand side.
\end{proof}

The last lemma of this section provides a crude estimate for the quantity $K_{a,b}(X)$ of 
use at the end of our iterative process.

\begin{lemma}\label{lemma3.3} Suppose that $a$ and $b$ are integers with 
$0\le a<b\le (3\tet)^{-1}$. Then provided that $\Lam\ge 0$, one has
$$\llbracket K_{a,b}(X)\rrbracket \ll X^{\Lam+\del}(M^{b-a})^{k(k+1)/2}.$$
\end{lemma}

\begin{proof} Let $\xi$ and $\eta$ be integers with $1\le \xi\le \varpi^c$ and 
$1\le \eta\le \varpi^d$. Then, as in the discussion of \S2, we find that the mean value 
$K_{a,b}(X;\xi,\eta)$ counts the number of integral solutions $\bfx,\bfy,\bfv,\bfw$ of the 
system (\ref{2.w3}) subject to $|\bfx|,|\bfy|,|\bfv|,|\bfw|\le X$,
\begin{equation}\label{3.6}
\bfx,\bfy\in \Xi_a(\xi)\mmod{\varpi^{a+1}},\quad \bfv_l,\bfw_l\in \Xi_b(\eta)
\mmod{\varpi^{b+1}}\quad (1\le l\le u),
\end{equation}
each solution being counted with weight (\ref{2.w4}). Since we suppose the weights 
$\gra_n$ to be positive, we may relax the conditions (\ref{3.6}) to insist only that
$$\bfx\equiv \bfy\equiv \xi\mmod{\varpi^a},\quad \bfv_l\equiv\bfw_l\equiv \eta
\mmod{\varpi^b}\quad (1\le l\le u),$$
this relaxation only increasing our resulting estimate for $K_{a,b}(X;\xi,\eta)$. By 
reinterpreting the associated number of solutions of the system (\ref{2.w3}) via 
orthogonality and invoking Lemma \ref{lemma3.2}, we deduce that
\begin{align*}
K_{a,b}(X;\xi,\eta)&\le \oint |\grf_a(\bfalp;\xi)^{2k}\grf_b(\bfalp;\eta)^{2s}|\d\bfalp \\
&\ll \left( (X/M^a)^k(X/M^b)^s\right)^{(\lam+\del)/(s+k)}.
\end{align*}

\par Next we recall (\ref{2.18}), and note that the definition (\ref{2.10}) ensures that
$$\sum_{\xi=1}^{\varpi^a}\rho_a(\xi)^2=\rho_0(1)^2=\sum_{\eta=1}^{\varpi^b}
\rho_b(\eta)^2.$$
In view of (\ref{2.23}), we deduce that
$$K_{a,b}(X)\ll (X/M^a)^{k-k(k+1)/2}(X/M^b)^s(M^{s(b-a)/(s+k)})^{k(k+1)/2}
X^{\Lam+\del}.$$
We consequently conclude from (\ref{2.21}) that
$$\llbracket K_{a,b}(X)\rrbracket \ll (M^{s(b-a)/(s+k)})^{k(k+1)/2}X^{\Lam+\del},$$
which suffices to complete the proof of the lemma.
\end{proof}

\section{The conditioning process} The variables underlying the mean value 
$I_{a,b}(X;\xi,\eta)$ must be conditioned so as to ensure that appropriate non-singularity 
conditions hold, yielding the mean value $K_{a,b}(X;\xi,\eta)$ central to the congruencing 
process. This we achieve in the next two lemmata. We note in this context that our 
conditioning treatment here is considerably sharper than in our earlier work associated with 
Vinogradov's mean value theorem.

\begin{lemma}\label{lemma4.1} Let $a$ and $b$ be integers with $b>a\ge 0$. Then one 
has
$$I_{a,b}(X)\ll K_{a,b}(X)+I_{a,b+1}(X).$$
\end{lemma}

\begin{proof}Consider fixed integers $\xi$ and $\eta$ with $1\le \xi\le \varpi^a$ and 
$1\le \eta\le \varpi^b$. Then by orthogonality, one finds from (\ref{2.15}) that $I_{a,b}
(X;\xi,\eta)$ counts the number of integral solutions $\bfx,\bfy,\bfv,\bfw$ of the system 
(\ref{2.w1}), with its attendant conditions, and with each solution counted with weight 
(\ref{2.w2}). Let $T_1(\xi,\eta)$ denote the contribution to $I_{a,b}(X;\xi,\eta)$ arising 
from those integral solutions in which three at least of the integers $v_1,\ldots ,v_s$ lie in 
a common congruence class modulo $\varpi^{b+1}$. Also, let $T_2(\xi,\eta)$ denote the 
corresponding contribution to $I_{a,b}(X;\xi,\eta)$ arising from those integral solutions in 
which no three of the integers $v_1,\ldots ,v_s$ lie in a common congruence class modulo 
$\varpi^{b+1}$. Thus we have
\begin{equation}\label{4.1}
I_{a,b}(X;\xi,\eta)=T_1(\xi,\eta)+T_2(\xi,\eta).
\end{equation}

\par We begin by estimating the quantity $T_1(\xi,\eta)$. Let $\bfx,\bfy,\bfv,\bfw$ be a 
solution contributing to $T_1(\xi,\eta)$. By relabelling the suffices of the variables 
$v_1,\ldots ,v_s$, we may suppose that $v_1\equiv v_2\equiv v_3\mmod{\varpi^{b+1}}$, 
provided that we inflate the ensuing estimates by a factor $\binom{s}{3}$. Then, on 
recalling the presumed positivity of the weights $\gra_n$, it follows via orthogonality that
\begin{equation}\label{4.2}
T_1(\xi,\eta)\ll \oint |\grF_a(\bfalp;\xi)^2\grG_b(\bfalp;\eta)\grf_b(\bfalp;\eta)^{2s-3}|
\d\bfalp ,
\end{equation}
in which we write
\begin{equation}\label{4.3}
\grG_b(\bfalp;\eta)=\rho_b(\eta)^{-3}\biggl( \sum_{\substack{1\le \zet\le 
\varpi^{b+1}\\ 
\zet\equiv \eta\mmod{\varpi^b}}}\rho_{b+1}(\zet)\grf_{b+1}(\bfalp;\zet)\biggr)^3.
\end{equation}
We note in this context that our assumption $s\ge k^2$ ensures that $s\ge 3$. In view of 
(\ref{2.15}), an application of H\"older's inequality on the right hand side of (\ref{4.2}) 
yields the bound
$$T_1(\xi,\eta)\ll I_{a,b}(X;\xi,\eta)^{1-3/(2s)}\biggl( \oint |\grF_a(\bfalp;\xi)^2
\grG_b(\bfalp;\eta)^{2s/3}|\d\bfalp \biggr)^{3/(2s)}.$$
Thus, now referring to (\ref{2.17}), a second application of H\"older's inequality leads 
from (\ref{4.1}) to the estimate
$$I_{a,b}(X)\ll I_{a,b}(X)^{1-3/(2s)}T_3^{3/(2s)}+\rho_0(1)^{-4}
\sum_{\xi=1}^{\varpi^a}\sum_{\eta=1}^{\varpi^b}\rho_a(\xi)^2\rho_b(\eta)^2
T_2(\xi,\eta),$$
where
\begin{equation}\label{4.4}
T_3=\rho_0(1)^{-4}\sum_{\xi=1}^{\varpi^a}\sum_{\eta=1}^{\varpi^b}\rho_a(\xi)^2
\rho_b(\eta)^2\oint |\grF_a(\bfalp;\xi)^2\grG_b(\bfalp;\eta)^{2s/3}|\d\bfalp .
\end{equation}
We therefore deduce that
\begin{equation}\label{4.5}
I_{a,b}(X)\ll T_3+\rho_0(1)^{-4}\sum_{\xi=1}^{\varpi^a}\sum_{\eta=1}^{\varpi^b}
\rho_a(\xi)^2\rho_b(\eta)^2T_2(\xi,\eta).
\end{equation}

\par In order to estimate $T_3$, we begin with an application of H\"older's inequality to 
(\ref{4.3}), obtaining
$$|\grG_b(\bfalp;\eta)|^{2s/3}\le \left( \grH_b^{(1)}(\bfalp;\eta)\right)^{s(2s-3)/(3s-3)}
\left( \grH_b^{(s)}(\bfalp;\eta)\right)^{s/(3s-3)},$$
where we write
\begin{equation}\label{4.6}
\grH_b^{(t)}(\bfalp;\eta)=\rho_b(\eta)^{-2t}\sum_{\substack{1\le \zet\le \varpi^{b+1}\\
\zet\equiv \eta\mmod{\varpi^b}}}\rho_{b+1}(\zet)^{2t}|\grf_{b+1}(\bfalp;\zet)|^{2t}.
\end{equation}
A further application of H\"older's inequality consequently conveys us from (\ref{4.4}) to 
the bound
\begin{equation}\label{4.6a}
T_3\le T_4^{(2s-3)/(3s-3)}T_5^{s/(3s-3)},
\end{equation}
in which
\begin{equation}\label{4.7}
T_4=\rho_0(1)^{-4}\sum_{\xi=1}^{\varpi^a}\sum_{\eta=1}^{\varpi^b}
\rho_a(\xi)^2\rho_b(\eta)^2\oint |\grF_a(\bfalp;\xi)^2\grH_b^{(1)}(\bfalp;\eta)^s|
\d\bfalp 
\end{equation}
and
\begin{equation}\label{4.8}
T_5=\rho_0(1)^{-4}\sum_{\xi=1}^{\varpi^a}\sum_{\eta=1}^{\varpi^b}
\rho_a(\xi)^2\rho_b(\eta)^2\oint |\grF_a(\bfalp;\xi)^2\grH_b^{(s)}(\bfalp;\eta)|
\d\bfalp .
\end{equation}

\par The integral within the definition (\ref{4.7}) of $T_4$ counts the number of integral 
solutions of the system (\ref{2.w1}), subject to its attendant conditions, with weight 
(\ref{2.w2}), and subject to the additional condition 
$v_l\equiv w_l\mmod{\varpi^{b+1}}$ $(1\le l\le s)$. Since the weights $\gra_n$ are 
presumed positive, the omission of this last condition merely inflates our estimate for this 
integral, and thus we see that it is bounded above by $I_{a,b}(X;\xi,\eta)$. We thus deduce 
that
\begin{equation}\label{4.9}
T_4\le I_{a,b}(X).
\end{equation}

\par In order to estimate $T_5$, we again recall that the weights $\gra_n$ are presumed 
positive. Thus, when $\zet\equiv \eta\mmod{\varpi^b}$, we have the upper bound
$$\rho_{b+1}(\zet)^2=\sum_{\substack{|n|\le X\\ n\equiv \zet\mmod{\varpi^{b+1}}}}
|\gra_n|^2\le \sum_{\substack{|n|\le X\\ n\equiv \eta\mmod{\varpi^b}}}|\gra_n|^2
=\rho_b(\eta)^2,$$
so that (\ref{4.6}) yields the bound
$$\grH_b^{(s)}(\bfalp;\eta)\le \rho_b(\eta)^{-2}\sum_{\substack{
1\le \zet\le \varpi^{b+1}\\ \zet\equiv \eta\mmod{\varpi^b}}}\rho_{b+1}(\zet)^2
|\grf_{b+1}(\bfalp;\zet)|^{2s}.$$
On substituting this bound into (\ref{4.8}), we infer from (\ref{2.15}) that
\begin{align*}
T_5&\le \rho_0(1)^{-4}\sum_{\xi=1}^{\varpi^a}\sum_{\eta=1}^{\varpi^b}
\sum_{\substack{1\le \zet\le \varpi^{b+1}\\ \zet\equiv \eta\mmod{\varpi^b}}}
\rho_a(\xi)^2\rho_{b+1}(\zet)^2I_{a,b+1}(X;\xi,\zet)\\
&=\rho_0(1)^{-4}\sum_{\xi=1}^{\varpi^a}\sum_{\zet=1}^{\varpi^{b+1}}
\rho_a(\xi)^2\rho_{b+1}(\zet)^2I_{a,b+1}(X;\xi,\zet).
\end{align*}
On recalling (\ref{2.17}), we therefore see that $T_5\le I_{a,b+1}(X)$, and thus we 
deduce from (\ref{4.6a}) and (\ref{4.9}) that
$$T_3\le (I_{a,b}(X))^{(2s-3)/(3s-3)}(I_{a,b+1}(X))^{s/(3s-3)}.$$
Substituting this bound into (\ref{4.5}) and disentangling the result, we conclude thus far 
that
\begin{equation}\label{4.10}
I_{a,b}(X)\ll I_{a,b+1}(X)+\rho_0(1)^{-4}\sum_{\xi=1}^{\varpi^a}
\sum_{\eta=1}^{\varpi^b}\rho_a(\xi)^2\rho_b(\eta)^2T_2(\xi,\eta).
\end{equation}

\par Now is the moment to estimate the contribution of $T_2(\xi,\eta)$. Let 
$\bfx,\bfy,\bfv,\bfw$ be a solution contributing to $T_2(\xi,\eta)$. Since $s\ge k^2\ge 2k$, 
and no three of the integers $v_1,\ldots ,v_s$ lie in a common congruence class modulo 
$\varpi^{b+1}$, it follows that $v_1,\ldots ,v_s$ together occupy at least $k$ distinct 
congruence classes modulo $\varpi^{b+1}$. By relabelling the suffices of the variables 
$v_1,\ldots ,v_s$, we may suppose that $v_1,\ldots ,v_k$ lie in distinct congruence classes 
modulo $\varpi^{b+1}$, provided that we inflate the ensuing estimates by a factor 
$\binom{s}{k}$. Then, again recalling that the weights $\gra_n$ may be assumed positive, 
it follows via orthogonality that
$$T_2(\xi,\eta)\ll \oint |\grF_a(\bfalp;\xi)|^2\grF_b(\bfalp;\eta)\grf_b(\bfalp;\eta)^{s-k}
\grf_b(-\bfalp;\eta)^s\d\bfalp .$$
By reference to (\ref{2.15}) and (\ref{2.16}), an application of H\"older's inequality shows 
that
$$T_2(\xi,\eta)\ll I_{a,b}(X;\xi,\eta)^{1-k/(2s)}K_{a,b}(X;\xi,\eta)^{k/(2s)}.$$
Here we recall that $s=uk$. Thus, appealing to H\"older's inequality yet again, we conclude 
via (\ref{2.17}) and (\ref{2.18}) that
$$\rho_0(1)^{-4}\sum_{\xi=1}^{\varpi^a}\sum_{\eta=1}^{\varpi^b}\rho_a(\xi)^2
\rho_b(\eta)^2T_2(\xi,\eta)\ll I_{a,b}(X)^{1-k/(2s)}K_{a,b}(X)^{k/(2s)}.$$
On substituting this bound into (\ref{4.10}), we obtain
$$I_{a,b}(X)\ll I_{a,b+1}(X)+I_{a,b}(X)^{1-k/(2s)}K_{a,b}(X)^{k/(2s)},$$
and the conclusion of the lemma follows by disentangling.
\end{proof}

By iterating Lemma \ref{lemma4.1}, we are able to estimate $I_{a,b}(X)$ in terms of 
conditioned mean values of type $K_{a,b+h}(X)$ $(h\ge 0)$.

\begin{lemma}\label{lemma4.2} Let $a$ and $b$ be integers with $0\le a<b$, and put 
$H=4(b-a)$. Suppose that $b+H\le (3\tet)^{-1}$. Then there exists an integer $h$ with 
$0\le h<H$ having the property that
$$I_{a,b}(X)\ll K_{a,b+h}(X)+M^{-sH/4}(X/M^b)^s(X/M^a)^{k-k(k+1)/2+\Lam}.$$
\end{lemma}

\begin{proof} By repeated application of Lemma \ref{lemma4.1}, we obtain the bound
\begin{equation}\label{4.w1}
I_{a,b}(X)\ll \sum_{h=0}^{H-1}K_{a,b+h}(X)+I_{a,b+H}(X).
\end{equation}
Let $\xi$ and $\eta$ be fixed integers with $1\le \xi\le \varpi^a$ and $1\le \eta \le 
\varpi^{b+H}$. Then on recalling our presumption that the weights $\gra_n$ are positive, 
it follows from (\ref{2.15}) via orthogonality that
$$I_{a,b+H}(X;\xi,\eta)\le \oint |\grf_a(\bfalp;\xi)^{2k}\grf_{b+H}(\bfalp;\eta)^{2s}|
\d\bfalp .$$
We therefore deduce from Lemma \ref{lemma3.2} that
\begin{align}
I_{a,b+H}(X;\xi,\eta)&\ll \left( (X/M^a)^k(X/M^{b+H})^s\right)^{(\lam+\del)/(s+k)}
\notag \\
&\ll X^\del (X/M^a)^{\lam-s}(X/M^b)^sM^\Ome ,\label{4.w3}
\end{align}
in which
$$\Ome =\lam\left( a-\frac{ak}{s+k}-\frac{bs}{s+k}\right) +s(b-a)-\frac{Hs\lam}{s+k}.
$$
We may suppose that $\lam\ge s+k-k(k+1)/2\ge \tfrac{1}{2}(s+k)$. Then since 
$b\ge a$, we obtain the estimate
$$\Ome \le -\frac{s(b-a)\lam}{s+k}+s(b-a)-\tfrac{1}{2}Hs\le \tfrac{1}{2}s(b-a-H).$$
But $H=4(b-a)$, and so we discern from (\ref{2.7}) that
$$\Ome\le -\tfrac{3}{8}Hs\le -\del \tet^{-1}-\tfrac{1}{4}Hs.$$
Substituting this estimate into (\ref{4.w3}), we see that
$$I_{a,b+H}(X;\xi,\eta)\ll M^{-sH/4}(X/M^a)^{\lam-s}(X/M^b)^s.$$

\par We next recall (\ref{2.17}), deducing that
$$I_{a,b+H}(X)\ll \rho M^{-sH/4}(X/M^a)^{\lam-s}(X/M^b)^s,$$
where
$$\rho=\rho_0(1)^{-4}\biggl( \sum_{\xi=1}^{\varpi^a}\rho_a(\xi)^2\biggr) 
\biggl( \sum_{\eta=1}^{\varpi^{b+H}}\rho_{b+H}(\eta)^2\biggr)=1.$$
The conclusion of the lemma consequently follows from (\ref{4.w1}).
\end{proof}

We next introduce a lemma that initiates the iterative process.

\begin{lemma}\label{lemma4.3} There exists a prime $\varpi$ with $M<\varpi\le 2M$ for 
which
$$U(X)\ll M^sI_{0,1}(X).$$
\end{lemma}

\begin{proof} By orthogonality, it follows from (\ref{2.3}) that $U(X)$ counts the number of 
integral solutions $\bfx,\bfy$ of the system
\begin{equation}\label{4.11}
\sum_{i=1}^{s+k}(x_i^j-y_i^j)=0\quad (1\le j\le k),
\end{equation}
with $|\bfx|,|\bfy|\le X$, each solution being counted with weight
$$\rho_0(1)^{-2s-2k}\prod_{i=1}^{s+k}\gra_{x_i}\gra_{y_i}.$$
Let $T_0$ denote the contribution of such solutions in which $x_i=x_j$ for some $i$ and $j$ 
with $1\le i<j\le k$, and let $T_1$ denote the corresponding contribution with $x_i=x_j$ 
for no $i$ and $j$ with $1\le i<j\le k$. Then we have
\begin{equation}\label{4.13}
U(X)=T_0+T_1.
\end{equation}

\par We recall again that the weights $\gra_n$ are presumed positive. Then by relabelling 
the suffices of $x_1,\ldots ,x_k$, we find that $T_0$ is bounded above by 
$\binom{k}{2}T_2$, where $T_2$ denotes the number of solutions of the system
$$2x_1^j+\sum_{i=3}^{s+k}x_i^j=\sum_{l=1}^{s+k}y_l^j\quad (1\le j\le k),$$
with $|\bfx|,|\bfy|\le X$, each solution being counted with weight
$$\rho_0(1)^{-2s-2k}\gra_{x_1}^2\biggl( \prod_{i=3}^{s+k}\gra_{x_i}\biggr) \biggl( 
\prod_{l=1}^{s+k}\gra_{y_l}\biggr) .$$
Put $\grb_n=\gra_n^2$ for $|n|\le X$. Then on recalling (\ref{1.2}), it follows via 
orthogonality and the triangle inequality that
$$T_0\ll \rho_0(1)^{-2}\oint |f_\bfgrb(2\bfalp;X)\ftil_\bfgra(\bfalp;X)^{2s+2k-2}|
\d\bfalp .$$
A second application of the triangle inequality reveals that
$$|f_\bfgrb(2\bfalp;X)|\le \sum_{|n|\le X}|\grb_n|=
\sum_{|n|\le X}|\gra_n|^2=\rho_0(1)^2.$$
Thus, an application of H\"older's inequality gives the bound
$$T_0\ll \biggl( \oint |\ftil_\bfgra (\bfalp;X)|^{2s+2k}\d\bfalp \biggr)^{1-1/(s+k)}
=U(X)^{1-1/(s+k)}.$$
On substituting this estimate into (\ref{4.13}) and disentangling, we deduce that
\begin{equation}\label{4.14}
U(X)\ll 1+T_1.
\end{equation}

\par Consider next a solution $\bfx,\bfy$ of (\ref{4.11}) contributing to $T_1$. Write
$$\Del(\bfx)=\prod_{1\le i<j\le k}|x_i-x_j|,$$
and note that $0<\Del(\bfx)<X^{k(k-1)}$. Let $\calP$ denote the set consisting of the 
smallest $[k^3/\tet]+1$ prime numbers exceeding $M$. It follows from the Prime Number 
Theorem that none of these primes exceed $2M$. Moreover, one has
$$\prod_{p\in \calP}p>M^{k^3/\tet}=X^{k^3}>\Del(\bfx),$$
and hence one at least of the primes belonging to $\calP$ does not divide $\Del(\bfx)$. In 
particular, there is a prime $\varpi\in \calP$ for which $x_i\equiv x_j\mmod{\varpi}$ for 
no $i$ and $j$ with $1\le i<j\le k$. Again recalling that the weights $\gra_n$ are presumed 
positive, it follows from (\ref{2.14}) via orthogonality that
$$T_1\ll \sum_{\varpi\in \calP}\oint \grF_0(\bfalp;1)\ftil_\bfgra(\bfalp;X)^s
\ftil_\bfgra(-\bfalp;X)^{s+k}\d\bfalp .$$
Therefore, as a consequence of Schwarz's inequality, one finds via (\ref{2.3}) and 
(\ref{2.15}) that there exists a prime $\varpi\in \calP$ for which
\begin{align*}
T_1&\ll \biggl( \oint |\grF_0(\bfalp;1)^2\grf_0(\bfalp;1)^{2s}|\d\bfalp \biggr)^{1/2}
\biggl( \oint |\ftil_\bfgra (\bfalp;X)|^{2s+2k}\d\bfalp \biggr)^{1/2}\\
&=I_{0,0}(X;1,1)^{1/2}U(X)^{1/2}.
\end{align*}
Thus we conclude by means of (\ref{4.14}) that
$$U(X)\ll 1+I_{0,0}(X;1,1)^{1/2}U(X)^{1/2},$$
whence
\begin{equation}\label{4.15}
U(X)\ll I_{0,0}(X;1,1).
\end{equation}
Here, we have made use of a trivial lower bound for $I_{0,0}(X;1,1)$ obtained by 
considering the diagonal contribution in combination with the assumption that the 
coefficients $\gra_n$ are positive.\par

Next, we split the summation in (\ref{2.2}) into arithmetic progressions modulo $\varpi$. 
Thus we obtain
$$\grf_0(\bfalp;1)=\rho_0(1)^{-1}\sum_{\xi=1}^\varpi \rho_1(\xi)\grf_1(\bfalp;\xi).$$
By applying H\"older's inequality, we deduce that
$$|\grf_0(\bfalp;1)|^{2s}\le \rho_0(1)^{-2s}\biggl( \sum_{\xi=1}^\varpi \rho_1(\xi)^2
\biggr)^{s-1}\biggl( \sum_{\xi=1}^\varpi 1\biggr)^s\biggl( \sum_{\xi=1}^\varpi 
\rho_1(\xi)^2|\grf_1(\bfalp;\xi)|^{2s}\biggr) .$$
But
$$\sum_{\xi=1}^\varpi \rho_1(\xi)^2=\rho_0(1)^2,$$
and hence we deduce from (\ref{2.15}) and (\ref{2.17}) that
\begin{align*}
I_{0,0}(X;1,1)&\ll \rho_0(1)^{-2}M^s\sum_{\xi=1}^\varpi \rho_1(\xi)^2\oint 
|\grF_0(\bfalp;1)^2\grf_1(\bfalp;\xi)^{2s}|\d\bfalp \\
&=\rho_0(1)^{-2}M^s\sum_{\xi=1}^\varpi \rho_1(\xi)^2I_{0,1}(X;\xi,1)\\
&=M^sI_{0,1}(X).
\end{align*}
The conclusion of the lemma is now immediate from (\ref{4.15}).
\end{proof}

We now fix the prime $\varpi$, once and for all, in accordance with Lemma \ref{lemma4.3}. 
Finally, we obtain the starting point of our iteration in the next lemma.

\begin{lemma}\label{lemma4.4} There exists an integer $h\in \{0,1,2,3\}$ for which one 
has
$$U(X)\ll M^sK_{0,1+h}(X).$$
\end{lemma}

\begin{proof} According to Lemma \ref{lemma4.2}, there exists an integer $h$ with 
$0\le h<4$ having the property that
$$I_{0,1}(X)\ll K_{0,1+h}(X)+M^{-s}(X/M)^sX^{k-k(k+1)/2+\Lam}.$$
Since we may suppose that $M>X^\del$, it follows from Lemma \ref{lemma4.3} that
$$U(X)\ll M^sK_{0,1+h}(X)+X^{s+k-k(k+1)/2+\Lam-2\del}.$$
But in view of (\ref{2.22}), we have
$$U(X)>X^{s+k-k(k+1)/2+\Lam-\del},$$
and hence we arrive at the upper bound
$$U(X)\ll M^sK_{0,1+h}(X)+X^{-\del}U(X).$$
The conclusion of the lemma is now immediate.
\end{proof}

\section{The efficient congruencing step} The mean value $K_{a,b}(X;\xi,\eta)$ is subject 
to a powerful congruence condition on its underlying variables. In this section, we uncover 
this condition and convert it into one suitable for iteration. We begin with some preliminary 
discussion of congruences. Denote by $\calB_{a,b}(\bfm;\xi,\eta)$ the set of solutions of 
the system of congruences
\begin{equation}\label{5.1}
\sum_{i=1}^k(z_i-\eta)^j\equiv m_j\mmod{\varpi^{jb}}\quad (1\le j\le k),
\end{equation}
with $1\le \bfz\le \varpi^{kb}$ and $\bfz\equiv \bfxi\mmod{\varpi^{a+1}}$ for some 
$\bfxi\in \Xi_a(\xi)$.

\begin{lemma}\label{lemma5.1} Suppose that $a$ and $b$ are non-negative integers with 
$b>a$. Then
$$\max_{1\le \xi\le \varpi^a}\max_{1\le \eta\le \varpi^b}\text{card}\left( 
\calB_{a,b}(\bfm;\xi,\eta)\right)\le k!\varpi^{\frac{1}{2}k(k-1)(a+b)}.$$
\end{lemma}

\begin{proof} This is a special case of \cite[Lemma 4.1]{Woo2012}.
\end{proof}

\begin{lemma}\label{lemma5.2} Suppose that $a$ and $b$ are integers with $0\le 
a<b\le \tet^{-1}$. Then
$$K_{a,b}(X)\ll M^{\frac{1}{2}k(k-1)(b+a)}\left( I_{b,kb}(X)\right)^{k/s}
(X/M^b)^{(1-k/s)(\lam+\del)}.$$
\end{lemma}

\begin{proof} Let $\xi$ and $\eta$ be fixed integers with $1\le \xi\le \varpi^a$ and 
$1\le \eta\le \varpi^b$. We recall that $K_{a,b}(X;\xi,\eta)$ counts the number of integral 
solutions $\bfx,\bfy,\bfv,\bfw$ of the system (\ref{2.w3}), subject to its attendant 
conditions, with each solution counted with weight (\ref{2.w4}). Moreover, these solutions 
are subject to the congruence conditions (\ref{2.w5}). A comparison between the latter 
and (\ref{5.1}) shows that for each solution $\bfx,\bfy$ of the system of congruences 
(\ref{2.w5}), there is a $k$-tuple $\bfm$ for which one has both 
$\bfx\in \calB_{a,b}(\bfm;\xi,\eta)$ and $\bfy\in \calB_{a,b}(\bfm;\xi,\eta)$.\par

Write
$$\grG_{a,b}(\bfalp;\xi,\eta;\bfm)=\sum_{\bfzet\in \calB_{a,b}(\bfm;\xi,\eta)}
\prod_{i=1}^k\frac{\rho_{kb}(\zet_i)}{\rho_a(\xi)}\grf_{kb}(\bfalp;\zet_i).$$
Then it follows by orthogonality that
\begin{equation}\label{5.2}
K_{a,b}(X;\xi,\eta)=\sum_{m_1=1}^{\varpi^b}\ldots \sum_{m_k=1}^{\varpi^{kb}}
\oint |\grG_{a,b}(\bfalp;\xi,\eta;\bfm)^2\grF_b(\bfalp ;\eta)^{2u}|\d\bfalp .
\end{equation}
An application of Cauchy's inequality in combination with Lemma \ref{lemma5.1} yields the 
bound
\begin{align}
|\grG_{a,b}(\bfalp;\xi,\eta;\bfm)|^2&\le \text{card}(\calB_{a,b}(\bfm;\xi,\eta))
\sum_{\bfzet\in \calB_{a,b}(\bfm;\xi,\eta)}\prod_{i=1}^k\frac{\rho_{kb}(\zet_i)^2}
{\rho_a(\xi)^2}|\grf_{kb}(\bfalp;\zet_i)|^2\notag \\
&\ll M^{\frac{1}{2}k(k-1)(a+b)}\sum_{\bfzet\in 
\calB_{a,b}(\bfm;\xi,\eta)}\prod_{i=1}^k\frac{\rho_{kb}(\zet_i)^2}
{\rho_a(\xi)^2}|\grf_{kb}(\bfalp;\zet_i)|^2.\label{5.3}
\end{align}
On substituting (\ref{5.3}) into (\ref{5.2}), and again applying orthogonality, we deduce 
that
\begin{equation}\label{5.4}
K_{a,b}(X;\xi,\eta)\ll M^{\frac{1}{2}k(k-1)(a+b)}
\sum_{\substack{1\le \bfzet\le \varpi^{kb}\\ \bfzet\equiv \xi\mmod{\varpi^a}}}
\oint |\grF_b(\bfalp;\eta)|^{2u}
\prod_{i=1}^k\frac{\rho_{kb}(\zet_i)^2}{\rho_a(\xi)^2}|\grf_{kb}(\bfalp;\zet_i)|^2.
\end{equation}
Here, we have again made use of the presumed positivity of the coefficients $\gra_n$ in 
order to drop the implicit conditioning of the final block of $2k$ variables.\par

Next we observe that an application of H\"older's inequality reveals that
\begin{align*}
\sum_{\substack{1\le \bfzet\le \varpi^{kb}\\ \bfzet\equiv \xi\mmod{\varpi^a}}}
\prod_{i=1}^k\frac{\rho_{kb}(\zet_i)^2}{\rho_a(\xi)^2}|\grf_{kb}(\bfalp;\zet_i)|^2&
=\Biggl( \sum_{\substack{1\le \zet\le \varpi^{kb}\\ \zet\equiv \xi\mmod{\varpi^a}}}
\frac{\rho_{kb}(\zet)^2}{\rho_a(\xi)^2}|\grf_{kb}(\bfalp;\zet)|^2\Biggr)^k\\
&\le \rho^{k-1}\sum_{\substack{1\le \zet\le \varpi^{kb}\\ \zet\equiv 
\xi\mmod{\varpi^a}}}\frac{\rho_{kb}(\zet)^2}{\rho_a(\xi)^2}
|\grf_{kb}(\bfalp;\zet)|^{2k},
\end{align*}
where
$$\rho=\sum_{\substack{1\le \zet\le \varpi^{kb}\\ \zet\equiv \xi\mmod{\varpi^a}}}
\frac{\rho_{kb}(\zet)^2}{\rho_a(\xi)^2}=\rho_a(\xi)^{-2}\sum_{\substack{|n|\le X\\
n\equiv \xi\mmod{\varpi^a}}}|\gra_n|^2=1.$$
Thus we deduce from (\ref{5.4}) that
\begin{equation}\label{5.5}
K_{a,b}(X;\xi,\eta)\ll M^{\frac{1}{2}k(k-1)(a+b)}
\sum_{\substack{1\le \zet\le \varpi^{kb}\\ \zet\equiv \xi\mmod{\varpi^a}}}
\frac{\rho_{kb}(\zet)^2}{\rho_a(\xi)^2}V(\zet,\eta),
\end{equation}
in which we write
$$V(\zet,\eta)=\oint |\grf_{kb}(\bfalp;\zet)^{2k}\grF_b(\bfalp;\eta)^{2u}|\d\bfalp.$$

\par On recalling that $s=uk$, an application of H\"older's inequality delivers the bound
\begin{equation}\label{5.6}
V(\zet,\eta)\le V_1^{1-k/s}V_2^{k/s},
\end{equation}
where
$$V_1=\oint |\grF_b(\bfalp;\eta)|^{2u+2}\d\bfalp $$
and
$$V_2=\oint |\grF_b(\bfalp;\eta)^2\grf_{kb}(\bfalp;\zet)^{2s}|\d\bfalp .$$
By Lemma \ref{lemma3.1a}, one has $V_1\ll (X/M^b)^{\lam+\del}$. On the other hand, we 
find from (\ref{2.15}) that $V_2=I_{b,kb}(X;\eta,\zet)$. We therefore deduce from 
(\ref{5.5}) and (\ref{5.6}) via another application of H\"older's inequality that
$$K_{a,b}(X;\xi,\eta)\ll \rho_a(\xi)^{-2}M^{\frac{1}{2}k(k-1)(a+b)}
W_1^{1-k/s}W_2^{k/s},$$
where
$$W_1=\sum_{\substack{1\le \zet \le \varpi^{kb}\\ \zet\equiv \xi\mmod{\varpi^a}}}
\rho_{kb}(\zet)^2(X/M^b)^{\lam+\del}=\rho_a(\xi)^2(X/M^b)^{\lam+\del}$$
and
$$W_2=\sum_{\substack{1\le \zet\le \varpi^{kb}\\ \zet\equiv \xi\mmod{\varpi^a}}}
\rho_{kb}(\zet)^2I_{b,kb}(X;\eta,\zet).$$

\par Yet another application of H\"older's inequality yields
\begin{equation}\label{5.7}
\sum_{\xi=1}^{\varpi^a}\sum_{\eta=1}^{\varpi^b}\rho_a(\xi)^2\rho_b(\eta)^2
K_{a,b}(X;\xi,\eta)\ll M^{\frac{1}{2}k(k-1)(a+b)}Z_1^{1-k/s}Z_2^{k/s},
\end{equation}
where
$$Z_1=\sum_{\xi=1}^{\varpi^a}\sum_{\eta=1}^{\varpi^b}\rho_a(\xi)^2\rho_b(\eta)^2
(X/M^b)^{\lam+\del}=\rho_0(1)^4(X/M^b)^{\lam+\del},$$
and, in view of (\ref{2.18}),
\begin{align*}
Z_2&=\sum_{\eta=1}^{\varpi^b}\rho_b(\eta)^2\sum_{\xi=1}^{\varpi^a}
\sum_{\substack{1\le \zet\le \varpi^{kb}\\ \zet\equiv \xi\mmod{\varpi^a}}}
\rho_{kb}(\zet)^2I_{b,kb}(X;\eta,\zet)\\
&=\sum_{\eta=1}^{\varpi^b}\sum_{\zet=1}^{\varpi^{kb}}\rho_b(\eta)^2
\rho_{kb}(\zet)^2I_{b,kb}(X;\eta,\zet)=\rho_0(1)^4I_{b,kb}(X).
\end{align*}
An additional reference to (\ref{2.18}) therefore conveys us from (\ref{5.7}) to the bound
$$K_{a,b}(X)\ll \rho_0(1)^{-4}M^{\frac{1}{2}k(k-1)(b+a)}\left( \rho_0(1)^4I_{b,kb}(X)
\right)^{k/s}\left( \rho_0(1)^4(X/M^b)^{\lam+\del}\right)^{1-k/s},$$
and the proof of the lemma is complete.
\end{proof}

Finally, by combining the conclusion of Lemma \ref{lemma4.2} with Lemma 
\ref{lemma5.2}, we obtain the key bound for our iterative process.

\begin{lemma}\label{lemma5.3} Suppose that $a$ and $b$ are integers with 
$0\le a<b\le (12k\tet)^{-1}$. Put $H=4(k-1)b$. Then there exists an integer $h$, 
with $0\le h<H$, having the property that
$$\llbracket K_{a,b}(X)\rrbracket \ll X^\del M^{-kh}\llbracket K_{b,kb+h}(X)
\rrbracket^{k/s}(X/M^b)^{\Lam (1-k/s)}+M^{-kH/6}(X/M^b)^\Lam.$$
\end{lemma}

\begin{proof} Recall the definitions (\ref{2.20}) and (\ref{2.21}). Then it follows from 
Lemma \ref{lemma5.2} that
\begin{equation}\label{5.8}
\llbracket K_{a,b}(X)\rrbracket \ll X^\del M^\ome\llbracket I_{b,kb}(X)\rrbracket^{k/s}
(X/M^b)^{\Lam (1-k/s)},
\end{equation}
in which we have written
$$\ome=\tfrac{1}{2}k(k-1)(b+a)+(k-\tfrac{1}{2}k(k+1))(a-b)-k(k-1)b.$$
Since $\ome =0$, we may proceed to apply Lemma \ref{lemma4.2}. Thus, we deduce that 
there exists an integer $h$ with $0\le h<H$ having the property that
$$I_{b,kb}(X)\ll K_{b,kb+h}(X)+M^{-sH/4}(X/M^{kb})^s(X/M^b)^{k-k(k+1)/2+\Lam}.$$
Again referring to (\ref{2.20}) and (\ref{2.21}), we see that
$$\llbracket I_{b,kb}(X)\rrbracket \ll M^{-sh}\llbracket K_{b,kb+h}(X)\rrbracket
+(X/M^b)^\Lam M^{-sH/4}.$$
Substituting this estimate into (\ref{5.8}), we obtain the bound
$$\llbracket K_{a,b}(X)\rrbracket \ll X^\del M^{-kh}\llbracket K_{b,kb+h}(X)
\rrbracket^{k/s}(X/M^b)^{\Lam (1-k/s)}+X^\del M^{-kH/4}(X/M^b)^\Lam.$$
Since $M\ge X^{2\del}$, the conclusion of the lemma now follows.
\end{proof}

\section{The iterative process} The skeleton of our iterative process is now visible. Lemma 
\ref{lemma4.4} bounds $U(X)$ in terms of $K_{0,1+h}(X)$, whilst Lemma \ref{lemma5.3} 
bounds $K_{a,b}(X)$ in terms of $K_{b,kb+h}(X)$. Thus we may obtain a sequence of 
bounds for $U(X)$ in terms of auxiliary mean values $K_{a_n,b_n}(X)$, with $a_n$ and 
$b_n$ increasing with $n$. At any point in this iteration, we may apply the trivial bound 
for $K_{a,b}(X)$ supplied by Lemma \ref{lemma3.3}. It transpires that, with an 
appropriate choice of parameters (already selected within our argument), we arrive at a 
contradiction whenever $\Lam>0$. We begin by distilling Lemma \ref{lemma5.3} into a 
more portable form.

\begin{lemma}\label{lemma6.1} Suppose that $\Lam\ge 0$. Let $a,b\in \dbZ$ satisfy 
$0\le a<b\le (12k\tet)^{-1}$. In addition, suppose that there are real numbers $\psi$, 
$c$ and $\gam$, with
$$0\le c\le (2\del)^{-1}\tet,\quad \gam\ge 0\quad \text{and}\quad \psi\ge 0,$$
such that
\begin{equation}\label{6.1}
X^\Lam M^{\Lam \psi}\ll X^{c\del}M^{-\gam}\llbracket K_{a,b}(X)\rrbracket .
\end{equation}
Then, for some integer $h$ with $0\le h\le 4kb$, one has
$$X^\Lam M^{\Lam \psi'}\ll X^{c'\del}M^{-\gam'}\llbracket K_{b,kb+h}(X)\rrbracket ,$$
where
$$\psi'=(s/k)\psi+(s/k-1)b,\quad c'=(s/k)(c+1),\quad \gam'=(s/k)\gam+sh.$$
\end{lemma}

\begin{proof} By hypothesis, we have $X^{c\del}\le M$. We therefore deduce from 
Lemma \ref{lemma5.3} that there exists an integer $h$ with $0\le h<4kb$ having the 
property that
$$\llbracket K_{a,b}(X)\rrbracket \ll X^\del M^{-kh}
\llbracket K_{b,kb+h}(X)\rrbracket^{k/s}(X/M^b)^{\Lam (1-k/s)}+M^{-kH/6}
(X/M^b)^\Lam .$$
The hypothesised bound (\ref{6.1}) consequently leads to the estimate
$$X^\Lam M^{\Lam \psi}\ll X^{(c+1)\del}M^{-\gam-kh}\llbracket K_{b,kb+h}(X)
\rrbracket^{k/s}(X/M^b)^{\Lam (1-k/s)}+M^{-kH/6}(X/M^b)^\Lam ,$$
whence
$$X^{\Lam k/s}M^{\Lam (\psi+(1-k/s)b)}\ll X^{(c+1)\del}M^{-\gam -kh}
\llbracket K_{b,kb+h}(X)\rrbracket^{k/s}.$$
The conclusion of the lemma follows on raising left and right hand sides of the last 
inequality to the power $s/k$.
\end{proof}

We now come to the main act.

\begin{lemma}\label{lemma6.2} One has $\Lam\le 0$.
\end{lemma}

\begin{proof} Suppose first that we are able to establish the conclusion of the lemma for 
$s=k^2$. Then, applying the trivial estimate (\ref{2.X}), we find that whenever $s>k^2$, 
one has
\begin{align*}
U_{s+k}(X;\bfgra)&\le \Bigl( \sup_{\bfalp\in [0,1)^k}|\ftil_\bfgra (\bfalp;X)|
\Bigr)^{2(s-k^2)}\oint |\ftil_\bfgra (\bfalp;X)|^{2k(k+1)}\d\bfalp \\
&\ll X^{s-k^2}U_{k(k+1)}(X;\bfgra)\ll X^{s-k^2}\left( X^{k(k+1)/2+\eps}\right) .
\end{align*}
Hence
$$U_{s+k}(X;\bfgra)\ll X^{s+k-k(k+1)/2+\eps},$$
and we find that $\Lam\le 0$. We are therefore at liberty to assume in what follows that 
$s=k^2$.\par

We may now suppose that $s=k^2$ and $\Lam\ge 0$, for otherwise there is nothing to 
prove. We begin by applying Lemma \ref{lemma4.4}. Thus, in view of (\ref{2.19}) and 
(\ref{2.21}), there exists an integer $h_{-1}\in \{0,1,2,3\}$ such that
$$\llbracket U(X)\rrbracket \ll (M^{h_{-1}})^{-s}\llbracket K_{0,1+h_{-1}}(X)\rrbracket .
$$
We therefore deduce from (\ref{2.22}) that
\begin{equation}\label{6.2}
X^\Lam \ll X^\del \llbracket U(X)\rrbracket \ll X^\del M^{-sh_{-1}}\llbracket 
K_{0,1+h_{-1}}(X)\rrbracket .
\end{equation}

\par Next we define the sequences $(h_n)$, $(a_n)$, $(b_n)$, $(c_n)$, $(\psi_n)$ and 
$(\gam_n)$ for $0\le n\le N$, in such a way that
\begin{equation}\label{6.3}
0\le h_{n-1}\le 12kb_{n-1},
\end{equation}
and
\begin{equation}\label{6.4}
X^\Lam M^{\Lam \psi_n}\ll X^{c_n\del}M^{-\gam_n}\llbracket K_{a_n,b_n}(X)
\rrbracket .
\end{equation}
Given a fixed choice for the sequence $(h_n)$, the remaining sequences are defined by 
means of the relations
\begin{equation}\label{6.5}
a_{n+1}=b_n,\quad b_{n+1}=kb_n+h_n,
\end{equation}
\begin{align}
c_{n+1}&=k(c_n+1),\label{6.6}\\
\psi_{n+1}&=k\psi_n+(k-1)b_n,\label{6.7}\\
\gam_{n+1}&=k\gam_n+sh_n.\label{6.8}
\end{align}
We put
$$a_0=0,\quad b_{-1}=1,\quad b_0=1+h_{-1},$$
$$\psi_0=0,\quad c_0=1,\quad \gam_0=sh_{-1},$$
so that both (\ref{6.3}) and (\ref{6.4}) hold with $n=0$ as a consequence of our initial 
choice of $h_{-1}$ together with (\ref{6.2}). We prove by induction that for each 
non-negative integer $n$ with $n<N$, the sequence $(h_m)_{m=-1}^n$ may be chosen 
in such a way that
\begin{equation}\label{6.9}
0\le a_n\le b_n\le (12k\tet)^{-1},\quad b_n\ge ka_n,
\end{equation}
\begin{equation}\label{6.10}
\psi_n\ge 0,\quad \gam_n\ge 0,\quad 0\le c_n\le (2\del)^{-1}\tet ,
\end{equation}
and so that (\ref{6.3}) and (\ref{6.4}) both hold with $n$ replaced by $n+1$.\par

Let $0\le n\le N$, and suppose that (\ref{6.3}) and (\ref{6.4}) both hold for the index 
$n$. We have already shown such to be the case for $n=0$. We observe first that the 
relation (\ref{6.5}) demonstrates that $b_m\ge ka_m$ for all $m$. Also, from (\ref{6.3}) 
and (\ref{6.5}) one finds that $b_{m+1}\le 13kb_m$, whence
$$b_n\le (13k)^nb_0\le 4(13k)^n.$$
Thus we see from (\ref{2.7}) that $b_n\le (12k\tet)^{-1}$. Also, from (\ref{6.6}), 
(\ref{6.7}) and (\ref{6.8}) we see that $c_n$, $\psi_n$ and $\gam_n$ are non-negative 
for each $n$. Further, we have
\begin{equation}\label{6.11}
c_m\le k^m+k\Bigl( \frac{k^m-1}{k-1}\Bigr)\le 3k^m\quad (m\ge 0).
\end{equation}
In order to bound $\gam_n$, we begin by noting from (\ref{6.5}) that
$$h_m=b_{m+1}-kb_m\quad \text{and}\quad a_m=b_{m-1}.$$
Then it follows from (\ref{6.8}) that
$$\gam_{m+1}-sb_{m+1}=k(\gam_m-sb_m).$$
By iterating this relation, we deduce that for $m\ge 1$, one has
\begin{equation}\label{6.12}
\gam_m=sb_m+k^m(\gam_0-sb_0)=s(b_m-k^m).
\end{equation}

\par We may now suppose that (\ref{6.4}), (\ref{6.9}) and (\ref{6.10}) hold for the index 
$n$. An application of Lemma \ref{lemma6.1} therefore reveals that there exists an 
integer $h_n$ satisfying (\ref{6.3}) with $n$ replaced by $n+1$, for which the upper 
bound (\ref{6.4}) holds, also with $n$ replaced by $n+1$. This completes the inductive 
step, so that (\ref{6.4}) is now known to hold for $0\le n\le N$.\par

\par We now exploit the bound (\ref{6.4}) with $n=N$. Since $b_N\le 4(13k)^N\le 
(3\tet)^{-1}$, one finds from Lemma \ref{lemma3.3} that
$$\llbracket K_{a_N,b_N}(X)\rrbracket \ll X^{\Lam+\del}(M^{b_N-b_{N-1}})^{
k(k+1)/2}.$$
By combining this bound with (\ref{6.4}) and (\ref{6.12}), we obtain the estimate
\begin{align*}
X^\Lam M^{\Lam \psi_N}&\ll X^{\Lam+(c_N+1)\del}M^{\frac{1}{2}k(k+1)
(b_N-b_{N-1})-\gam_N}\\
&\ll X^{\Lam+(c_N+1)\del}M^{sk^N-(s-k(k+1)/2)b_N-\frac{1}{2}k(k+1)b_{N-1}}\\
&\ll X^{\Lam+(c_N+1)\del }M^{sk^N}.
\end{align*}
Meanwhile, from (\ref{6.11}) and (\ref{2.7}) we have $X^{(c_N+1)\del}<M$. We 
therefore deduce that
\begin{equation}\label{6.14}
\Lam \psi_N\le sk^N+1.
\end{equation}
Next, recalling that $b_m\ge k^m$ for each $m$, we deduce from (\ref{6.7}) that
$$\psi_{n+1}\ge k\psi_n+(k-1)k^n\quad (0\le n<N),$$
whence $\psi_N\ge N(k-1)k^{N-1}$. We thus conclude from (\ref{6.14}) that
$$\Lam\le \frac{sk^N+1}{N(k-1)k^{N-1}}\le \frac{4s}{N}.$$
Since we are at liberty to take $N$ as large as we please in terms of $s$ and $k$, we are 
forced to conclude that $\Lam\le 0$. In view of the discussion in the initial paragraph of 
the proof, this completes the proof of the lemma.
\end{proof}

The proof of the first estimate of Theorem \ref{theorem1.1} now follows. For the 
conclusion of Lemma \ref{lemma6.2} implies that when $s\ge k^2$, the bound
$$U_{s+k}(X;\bfgra)\ll X^{s+k-k(k+1)/2+\Lam+\eps}$$
holds with $\Lam=0$, and any $\eps>0$. Thus we deduce from (\ref{2.2}) and 
(\ref{2.3}) that whenever $s\ge k(k+1)$, then
$$\oint |f_\bfgra(\bfalp;X)|^{2s}\d\bfalp \ll X^{s-k(k+1)/2+\eps}\biggl( 
\sum_{|n|\le X}|\gra_n|^2\biggr)^s.$$
It remains for us to establish the second ($\eps$-free) estimate claimed by Theorem 
\ref{theorem1.1} when $s>k(k+1)$. This task we defer to the next section.

\section{The Keil-Zhao device} We now describe a relatively cheap method of slightly 
sharpening the first estimate of Theorem \ref{theorem1.1} when excess variables are 
present. This is motivated by recent work of Lilu Zhao \cite[equation (3.10)]{Zha2014} 
and Keil \cite[page 608]{Kei2014}. Kevin Hughes has also established such an 
$\eps$-removal lemma by an alternate and earlier route that has priority in this topic 
(see \cite{Hug2015}, and also \cite{Hen2015} for independent work on this topic). We 
include this section partly to highlight the utility of the Keil-Zhao device, and also to present a 
relatively self-contained account of the proof of Theorem \ref{theorem1.1}. The first results 
of this type are due to Bourgain \cite{Bou1989, Bou1993a}, and concern quadratic problems.
\par

The basic approach utilises the Hardy-Littlewood method. We therefore begin with some 
infrastructure. Write $L=X^{1/(2k)}$. Then, when $1\le q\le L$, $1\le a_j\le q$ $(1\le j\le k)$ and 
$(q,a_1,\ldots ,a_k)=1$, define the major arc $\grM(q,\bfa)$ by
$$\grM(q,\bfa)=\{\bfalp\in [0,1)^k:|\alp_j-a_j/q|\le LX^{-j}\quad (1\le j\le k)\}.$$
The arcs $\grM(q,\bfa)$ are disjoint, as is easily verified. Let $\grM$ denote their union, 
and put $\grm=[0,1)^k\setminus \grM$.\par

Write
$$F(\bfalp;X)=\sum_{|n|\le X}e(\alp_1n+\ldots +\alp_kn^k).$$
Also, when $\bfalp\in \grM(q,\bfa)\subseteq \grM$, write
$$V(\bfalp;q,\bfa)=q^{-1}S(q,\bfa)I(\bfalp-\bfa/q;X),$$
where
$$S(q,\bfa)=\sum_{r=1}^qe((a_1r+\ldots +a_kr^k)/q)$$
and
$$I(\bfbet;X)=\int_{-X}^Xe(\bet_1\gam+\ldots +\bet_k\gam^k)\d\gam .$$
We then define the function $V(\bfalp)$ to be $V(\bfalp;q,\bfa)$ when $\bfalp \in 
\grM(q,\bfa)\subseteq \grM$, and to be zero otherwise.\par

We make use of two basic estimates. The first follows from the argument of 
\cite[\S9]{Woo2012} with only trivial modifications, and shows that
\begin{equation}\label{7.0}
\sup_{\bfalp\in \grm}|F(\bfalp;X)|\ll X^{1-\tau+\eps},
\end{equation}
where $\tau^{-1}=4k^2$. The second estimate we record in the shape of a lemma.

\begin{lemma}\label{lemma7.1} Suppose that $u>\tfrac{1}{2}k(k+1)+2$. Then one has
$$\int_\grM |F(\bfalp ;X)|^u\d\bfalp \ll_u X^{u-k(k+1)/2}.$$
\end{lemma}

\begin{proof} It follows from \cite[Theorem 7.2]{Vau1997} that when $\bfalp\in 
\grM(q,\bfa)\subseteq \grM$, one has
$$F(\bfalp;X)-V(\bfalp;q,\bfa)\ll q+X|q\alp_1-a_1|+\ldots +X^k|q\alp_k-a_k|\ll L^2.$$
Thus we obtain
\begin{equation}\label{7.1}
\int_\grM|F(\bfalp;X)|^u\d\bfalp \ll \int_\grM (L^2)^u\d\bfalp +\int_\grM 
|V(\bfalp)|^u\d\bfalp .
\end{equation}
But $\text{mes}(\grM)\ll L^{2k+1}X^{-k(k+1)/2}$, and so
\begin{equation}\label{7.1w1}
\int_\grM (L^2)^u\d\bfalp \ll L^{2u+2k+1}X^{-k(k+1)/2}\ll X^{u-k(k+1)/2}.
\end{equation}
Meanwhile, one finds that
\begin{equation}\label{7.1w2}
\int_\grM|V(\bfalp)|^u\d\bfalp =\grS\grJ,
\end{equation}
where
$$\grJ=\int_\grB|I(\bfbet;X)|^u\d\bfalp \quad \text{and}\quad 
\grS=\sum_{1\le q\le L}\sum_{\substack{1\le \bfa\le q\\ (q,a_1,\ldots ,a_k)=1}}|q^{-1}
S(q,\bfa)|^u,$$
in which we write
$$\grB=[-LX^{-1},LX^{-1}]\times \ldots \times [-LX^{-k},LX^{-k}].$$

\par Since \cite[Theorem 1.3]{AKC2004} shows that the singular integral
$$\int_{\dbR^k}|I(\bfbet;1)|^{2s}\d\bfbet$$
converges for $2s>\tfrac{1}{2}k(k+1)+1$, we find via two changes of variables that
$$\grJ\le X^{u-k(k+1)/2}\int_{\dbR^k}|I(\bfbet;1)|^u\d\bfbet \ll X^{u-k(k+1)/2}.$$
Also, by reference to \cite[Theorem 2.4]{AKC2004}, one sees that the singular series
$$\sum_{q=1}^\infty \sum_{\substack{1\le \bfa\le q\\ (q,a_1,\ldots ,a_k)=1}}|q^{-1}
S(q,\bfa)|^{2s}$$
converges for $2s>\tfrac{1}{2}k(k+1)+2$, and hence
$$\grS\le \sum_{q=1}^\infty \sum_{\substack{1\le \bfa\le q\\ (q,a_1,\ldots ,a_k)=1}}|q^{-1}
S(q,\bfa)|^u\ll 1.$$
On substituting these estimates into (\ref{7.1w2}), we deduce that
$$\int_\grM|V(\bfalp)|^u\d\bfalp \ll X^{u-k(k+1)/2},$$
and hence the conclusion of the lemma follows from (\ref{7.1}) and (\ref{7.1w1}).
\end{proof}

We are now equipped to apply the Keil-Zhao device.

\begin{lemma}\label{lemma7.2} Suppose that $w$ is a real number with 
$w\ge \tfrac{1}{2}k(k+1)$ for which one has the estimate
\begin{equation}\label{7.2}
\oint |\ftil_\bfgra (\bfalp ;X)|^{2w}\d\bfalp \ll X^{w-k(k+1)/2+\eps}.
\end{equation}
Then whenever $s>\max \{ w, \tfrac{1}{2}k(k+1)+2\}$, one has
$$\oint |\ftil_\bfgra (\bfalp ;X)|^{2s}\d\bfalp \ll X^{s-k(k+1)/2}.$$
\end{lemma}

\begin{proof} The hypothesis of the statement of the lemma permits us to assume that 
$s=w+k\nu$, for some $\nu>0$. Let $\del$ be a positive number sufficiently small in terms 
of $k$ and $\nu$, and put
$$\grB=\{ \bfalp\in [0,1)^k:|\ftil_\bfgra(\bfalp;X)|>X^{1/2-\del}\}.$$
Then the upper bound (\ref{7.2}) ensures that
\begin{equation}\label{7.3}
\text{mes}(\grB)\le \left( X^{1/2-\del}\right)^{-2w}\oint |\ftil_\bfgra (\bfalp;X)|^{2w}
\d\bfalp \ll X^{2w\del -k(k+1)/2+\eps}.
\end{equation}
Putting $\grb=[0,1)^k\setminus \grB$, it follows that
\begin{align}
\int_\grb |\ftil_\bfgra (\bfalp;X)|^{2s}\d\bfalp &\le \Bigl( \sup_{\bfalp \in \grb}
|\ftil_\bfgra (\bfalp;X)|\Bigr)^{2k\nu}\oint |\ftil_\bfgra (\bfalp;X)|^{2w}\d\bfalp \notag \\
&\ll (X^{1/2-\del})^{2k\nu}X^{w-k(k+1)/2+\eps}\notag \\
&\ll X^{s-k(k+1)/2-\del k\nu}.\label{7.4}
\end{align}

\par We next consider the mean value
$$\Ups=\int_\grB|\ftil_\bfgra (\bfalp ;X)|^{2s}\d\bfalp .$$
We first rewrite $\Ups$ in the form
$$\Ups=\sum_{|\bfn|\le X}\grc(\bfn)\int_\grB |\ftil_\bfgra(\bfalp;X)|^{2s-2}
e(\Psi(\bfn;\bfalp))\d\bfalp ,$$ 
in which
$$\grc(\bfn)=\rho(\bfgra;X)^{-2}\gra_{n_1}{\overline \gra}_{n_2}$$
and
$$\Psi(\bfn;\bfalp)=\psi(n_1;\bfalp)-\psi(n_2;\bfalp).$$
Applying Cauchy's inequality, we deduce that
$$\Ups^2\le \rho \sum_{|\bfn|\le X}\int_\grB \int_\grB |\ftil_\bfgra 
(\bfalp;X)\ftil_\bfgra (\bfbet;X)|^{2s-2}e(\Psi(\bfn;\bfalp-\bfbet))\d\bfalp \d\bfbet ,$$
where
$$\rho=\rho(\bfgra;X)^{-4}\biggl( \sum_{|n|\le X}|\gra_n|^2\biggr)^{2}=1.$$
Thus we obtain the relation
\begin{equation}\label{7.5}
\Ups^2\le \int_\grB \int_\grB |F(\bfalp-\bfbet;X)^{2}\ftil_\bfgra (\bfalp;X)^{2s-2}
\ftil_\bfgra (\bfbet;X)^{2s-2}|\d\bfalp \d\bfbet .
\end{equation}

\par Our first observation concerning the integral on the right hand side of (\ref{7.5}) 
concerns the set of points $(\bfalp,\bfbet)\in \grB\times \grB$ for which $\bfalp-\bfbet 
\in \grm$. On applying the bound (\ref{7.0}), we deduce via a trivial inequality for 
$\ftil_\bfgra (\bftet;X)$ that
\begin{align*}
\underset{\bfalp-\bfbet\in \grm}{\int_\grB \int_\grB}|F(\bfalp-\bfbet;X)^2
&\ftil_\bfgra(\bfalp;X)^{2s-2}\ftil_\bfgra(\bfbet;X)^{2s-2}|\d\bfalp \d\bfbet \\
&\le \Bigl( \sup_{\bfalp-\bfbet \in \grm}|F(\bfalp-\bfbet;X)|\Bigr)^2(X^{1/2})^{4s-4}
\int_\grB\int_\grB\d\bfalp\d\bfbet \\
&\ll (X^{1-\tau+\eps})^2X^{2s-2}\left( \text{mes}(\grB)\right)^2.
\end{align*}
On applying (\ref{7.3}), therefore, we infer that
\begin{align}
\underset{\bfalp-\bfbet\in \grm}{\int_\grB \int_\grB}|F(\bfalp-\bfbet;X)^2
\ftil_\bfgra(\bfalp;X)^{2s-2}&\ftil_\bfgra(\bfbet;X)^{2s-2}|\d\bfalp \d\bfbet \notag \\
&\ll X^{2s-2\tau+\eps}(X^{2w\del -k(k+1)/2+\eps})^2\notag \\
&\ll X^{2s-k(k+1)-\tau}.\label{7.6}
\end{align}

\par On the other hand, by applying the trivial inequality
$$|z_1\ldots z_r|\le |z_1|^r+\ldots +|z_r|^r,$$
it follows that
\begin{align*}
|\ftil_\bfgra(\bfalp;X)^{2s-2}\ftil_\bfgra(\bfbet;X)^{2s-2}|\ll &\, 
|\ftil_\bfgra (\bfalp;X)^{2s}\ftil_\bfgra (\bfbet;X)^{2s-4}|\\
&+|\ftil_\bfgra (\bfalp;X)^{2s-4}\ftil_\bfgra (\bfbet;X)^{2s}|.
\end{align*}
Then by symmetry, we obtain the bound
\begin{align*}
\underset{\bfalp-\bfbet\in \grM}{\int_\grB \int_\grB}&|F(\bfalp-\bfbet;X)^2
\ftil_\bfgra(\bfalp;X)^{2s-2}\ftil_\bfgra(\bfbet;X)^{2s-2}|\d\bfalp \d\bfbet \notag \\
&\ll \underset{\bfalp-\bfbet\in \grM}{\int_\grB \int_\grB}|F(\bfalp-\bfbet;X)^2
\ftil_\bfgra(\bfalp;X)^{2s-4}\ftil_\bfgra(\bfbet;X)^{2s}|\d\bfalp \d\bfbet .
\end{align*}
An application of H\"older's inequality therefore reveals that
\begin{equation}\label{7.7}
\underset{\bfalp-\bfbet\in \grM}{\int_\grB \int_\grB}|F(\bfalp-\bfbet;X)^2
\ftil_\bfgra(\bfalp;X)^{2s-2}\ftil_\bfgra(\bfbet;X)^{2s-2}|\d\bfalp \d\bfbet \ll 
I_1^{2/s}I_2^{1-2/s},
\end{equation}
where
$$I_1=\underset{\bfalp-\bfbet\in \grM}{\int_\grB \int_\grB}|F(\bfalp-\bfbet;X)^s
\ftil_\bfgra(\bfbet;X)^{2s}|\d\bfalp \d\bfbet $$
and
$$I_2=\underset{\bfalp-\bfbet\in \grM}{\int_\grB \int_\grB}|\ftil_\bfgra(\bfalp;X)^{2s}
\ftil_\bfgra(\bfbet;X)^{2s}|\d\bfalp \d\bfbet .$$

\par By Lemma \ref{lemma7.1}, since $s>\tfrac{1}{2}k(k+1)+2$, we see that
$$I_1\ll \biggl( \int_\grM |F(\bftet;X)|^s\d\bftet \biggr) \biggl( \int_\grB |\ftil_\bfgra 
(\bfbet;X)|^{2s}\d\bfbet \biggr) \ll X^{s-k(k+1)/2}\Ups .$$
On the other hand,
$$I_2\le \biggl( \int_\grB |\ftil_\bfgra (\bfalp;X)|^{2s}\d\bfalp \biggr) \biggl( \int_\grB 
|\ftil_\bfgra (\bfbet;X)|^{2s}\d\bfbet \biggr)=\Ups^2.$$
Thus, we infer from (\ref{7.7}) that
\begin{align}
\underset{\bfalp-\bfbet\in \grM}{\int_\grB \int_\grB}&|F(\bfalp-\bfbet;X)^2
\ftil_\bfgra(\bfalp;X)^{2s-2}\ftil_\bfgra(\bfbet;X)^{2s-2}|\d\bfalp \d\bfbet \notag \\
&\ll (X^{s-k(k+1)/2}\Ups)^{2/s}(\Ups^2)^{1-2/s}\notag \\
&\ll X^{2-k(k+1)/s}\Ups^{2-2/s}.\label{7.8}
\end{align}
Combining (\ref{7.5}) and (\ref{7.6}) with (\ref{7.8}), we deduce that
$$\Ups^2\ll X^{2s-k(k+1)-\tau}+X^{2-k(k+1)/s}\Ups^{2-2/s},$$
whence
$$\int_\grB|\ftil_\bfgra (\bfalp;X)|^{2s}\d\bfalp =\Ups\ll X^{s-k(k+1)/2}.$$
Finally, we combine the last bound with (\ref{7.4}), obtaining the estimate
\begin{align*}
\oint |\ftil_\bfgra (\bfalp;X)|^{2s}\d\bfalp &=\int_\grB |\ftil_\bfgra (\bfalp;X)|^{2s}\d\bfalp 
+\int_\grb |\ftil_\bfgra (\bfalp;X)|^{2s}\d\bfalp \\
&\ll X^{s-k(k+1)/2-\del k\nu}+X^{s-k(k+1)/2}\ll X^{s-k(k+1)/2}.
\end{align*}
This completes the proof of the lemma.
\end{proof}

Since the first conclusion of Theorem \ref{theorem1.1}, already established, confirms the validity 
of the hypothesis (\ref{7.2}) with $w=k(k+1)$, the second conclusion of Theorem 
\ref{theorem1.1} is immediate from Lemma \ref{lemma7.2}. Hence Theorem \ref{theorem1.1} 
has now been proved in full.

\section{Some consequences of Theorem \ref{theorem1.1}} We take some space in this 
section to provide brief accounts of the proofs of the two corollaries to Theorem 
\ref{theorem1.1}. This is all quite standard, and so we feel justified in economising on 
detail.

\begin{proof}[The proof of Corollary \ref{corollary1.2}] It follows from Theorem 
\ref{theorem1.1} that whenever $k\ge 2$, $p\ge 2k(k+1)$ and $\eps>0$, then for 
any complex sequence $(\gra_n)_{n\in \dbZ}$, one has
$$\oint |f_\bfgra (\bfalp;N)|^p\d\bfalp \ll N^{(p-k(k+1))/2+\eps}\biggl( \sum_{|n|\le N}
|\gra_n|^2\biggr)^{p/2}.$$
The estimate (\ref{1.3}) therefore follows on raising left and right hand sides of this 
relation to the power $1/p$, wherein
$$K_{p,N}\ll (N^{(p-k(k+1))/2+\eps})^{1/p}=N^{(1-\tet)/2+\eps}.$$
Here, one may take $\eps=0$ whenever $p>2k(k+1)$. The first claim of Corollary 
\ref{corollary1.2} has therefore been established.\par

We now address the bound for $A_{p,N}$. Let $g:(\dbR/\dbZ)^k\rightarrow \dbC$ have 
Fourier series defined as in (\ref{1.1}), and put
$$g_0(\bfalp)=\sum_{|n|\le N}\ghat(n,n^2,\ldots ,n^k)e(\alp_1n+\ldots +\alp_kn^k).$$
Then, by orthogonality in combination with H\"older's inequality, one has
\begin{align}
\sum_{|n|\le N}|\ghat(n,n^2,\ldots ,n^k)|^2&=\oint g_0(\bfalp)\gtil(-\bfalp)\d\bfalp 
\notag \\
&\le \biggl( \oint |g_0(\bfalp)|^p\d\bfalp \biggr)^{1/p}\biggl( \oint 
|\gtil(\bfalp)|^{p'}\d\bfalp \biggr)^{1/p'},\label{8.1}
\end{align}
in which we write $p'=p/(p-1)$. The first integral on the right hand side may be bounded 
via Theorem \ref{theorem1.1}, giving
$$\oint |g_0(\bfalp)|^p\d\bfalp \ll N^{(p-k(k+1))/2+\eps}\biggl( \sum_{|n|\le N}
|\ghat(n,n^2,\ldots ,n^k)|^2\biggr)^{p/2}.$$
On substituting this bound into (\ref{8.1}), we obtain the relation
$$\sum_{|n|\le N}|\ghat(n,n^2,\ldots ,n^k)|^2\ll N^{(1-\tet)/2+\eps}
\biggl( \sum_{|n|\le N}|\ghat(n,n^2,\ldots ,n^k)|^2\biggr)^{1/2}\|\gtil\|_{p'}.$$
On disentangling this relation, and noting that $\|\gtil\|_{p'}=\|g\|_{p'}$, we obtain the 
bound
$$\sum_{|n|\le N}|\ghat(n,n^2,\ldots ,n^k)|^2\ll A_{p,N}\|g\|_{p'}^2,$$
with $A_{p,N}\ll N^{1-\tet+\eps}$. The second conclusion of Corollary \ref{corollary1.2} 
follows, noting that one may take $\eps=0$ when $p>2k(k+1)$.
\end{proof}

\begin{proof}[The proof of Corollary \ref{corollary1.3}] Consider positive integers 
$k_1,\ldots ,k_t$ with $1\le k_1<k_2<\ldots <k_t=k$, and denote by $l_1,\ldots ,l_u$ 
those positive integers with $1\le l_1<l_2<\ldots <l_u<k$ for which
$$\{k_1,\ldots ,k_t\}\cup \{l_1,\ldots ,l_u\}=\{1,2,\ldots ,k\}.$$
Plainly, therefore, one has $u=k-t$, and moreover,
\begin{equation}\label{8.2}
\sum_{i=1}^ul_i=\sum_{j=1}^k j-\sum_{m=1}^tk_m=\tfrac{1}{2}k(k+1)-K.
\end{equation}
By orthogonality, the mean value
$$I=\oint \biggl| \sum_{|n|\le N}\gra_ne(\alp_1n^{k_1}+\ldots +\alp_tn^{k_t})
\biggr|^{2s}\d\bfalp $$
counts the number of integral solutions of the system of equations
$$\sum_{i=1}^s (x_i^{k_m}-y_i^{k_m})=0\quad (1\le m\le t),$$
with $|\bfx|,|\bfy|\le N$, with each solution counted with weight
\begin{equation}\label{8.3}
\prod_{i=1}^s\gra_{x_i}{\overline \gra}_{y_i}.
\end{equation}

\par Given any one solution $\bfx,\bfy$ counted by $I$, there exist integers $h_j$ 
$(1\le j\le k)$ for which
\begin{equation}\label{8.4}
\sum_{i=1}^s(x_i^j-y_i^j)=h_j\quad (1\le j\le k).
\end{equation}
Indeed, one has $h_j=0$ whenever $j=k_m$ for some suffix $m$. It is evident, 
moreover, that $|h_j|\le 2sN^j$ $(1\le j\le k)$.\par

When $\bfh\in \dbZ^k$, consider the integral
$$I(\bfh)=\oint |f_\bfgra (\bfbet;N)|^{2s}e(-h_1\bet_1-\ldots -h_k\bet_k)\d\bfbet .$$
By orthogonality, this integral counts the integral solutions of the system of equations 
(\ref{8.4}) with $|\bfx|,|\bfy|\le N$, and with each solution counted with weight 
(\ref{8.3}). Thus we see that
$$I=\sum_{|h_{l_1}|\le 2sN^{l_1}}\ldots \sum_{|h_{l_u}|\le 2sN^{l_u}}I(\bfh),$$
in which we put $h_j=0$ whenever $j=k_m$. By the triangle inequality, it therefore 
follows that
\begin{align*}
I&\le \sum_{|h_{l_1}|\le 2sN^{l_1}}\ldots \sum_{|h_{l_u}|\le 2sN^{l_u}}I({\mathbf 0})
\\
&\ll N^{l_1+\ldots +l_u}\oint |f_\bfgra (\bfbet ;N)|^{2s}\d\bfbet .
\end{align*}
Thus we deduce from (8.2) and Theorem \ref{theorem1.1} that when $s\ge k(k+1)$, 
one has
\begin{align*}
I&\ll N^{\frac{1}{2}k(k+1)-K}\left( N^{s-\frac{1}{2}k(k+1)+\eps}\biggl( 
\sum_{|n|\le N}|\gra_n|^2\biggr)^s\right) \\
&\ll N^{s-K+\eps}\biggl( \sum_{|n|\le N}|\gra_n|^2\biggr)^s .
\end{align*}
Here, one may take $\eps=0$ when $s>k(k+1)$. This completes the proof of Corollary 
\ref{corollary1.3}. 
\end{proof}

\section{Wider applications of weighted efficient congruencing} As we have already noted 
in the introduction, the ideas of this paper can be transported to deliver weighted variants 
of any mean value estimate established via efficient congruencing methods. Indeed, the 
basic \cite{FW2014, Woo2013} and multigrade variants of efficient congruencing 
introduced in \cite{Woo2014a,Woo2014b,Woo2014c} may all be modified to accommodate 
the weighted setting appropriate for restriction theory. Although this task is not especially 
easy, by adapting these methods one may establish the Main Conjecture (recorded above 
as Conjecture \ref{conjecture1.4}) for
$$1\le s\le \tfrac{1}{2}k(k+1)-(\tfrac{1}{3}+o(1))k\quad (\text{$k$ large})$$
and
$$s\ge k(k-1)\quad (k\ge 3).$$
Moreover, one may take $\eps=0$ when $s>k(k-1)$. The last result, in particular, confirms 
the Main Conjecture in full for $k=3$. As we have noted in the introduction, the associated 
arguments are of sufficient complexity that, were we to establish them in full as the main 
thrust of this paper, we would obscure the basic principles of the weighted efficient 
congruencing method. Instead, we intend to provide complete accounts of the proofs of 
these conclusions as special cases of more general results in subsequent papers.\par

Such ideas also extend to multidimensional settings. Consider, for example, a system of 
polynomials $\bfF=(F_1,\ldots ,F_r)$, with $F_i\in \dbZ[x_1,\ldots ,x_d]$ $(1\le i\le r)$. 
Following \cite[\S2]{PPW2013}, we say that $\bfF$ is translation-dilation invariant if:
\vskip.0cm
\noindent (i) the polynomials $F_1,\ldots ,F_r$ are each homogeneous of positive degree, 
and\vskip.0cm
\noindent (ii) there exist polynomials
$$c_{jl}\in \dbZ[\xi_1,\ldots ,\xi_d]\quad (\text{$1\le j\le r$ and $0\le l\le j$}),$$
with $c_{jj}=1$ for $1\le j\le r$, having the property that whenever $\bfxi\in \dbZ^d$, then
$$F_j(\bfx+\bfxi)=c_{j0}(\bfxi)+\sum_{l=1}^jc_{jl}(\bfxi)F_l(\bfx)\quad (1\le j\le r).$$
It follows that the system
\begin{equation}\label{9.1}
\sum_{i=1}^s(F_j(\bfx_i)-F_j(\bfy_i))=0\quad (1\le j\le r),
\end{equation}
possesses an integral solution $\bfx,\bfy$ if and only if, for each $\bfxi\in \dbZ^d$ and 
$\lam\in \dbZ\setminus\{0\}$, one has
$$\sum_{i=1}^s(F_j(\lam \bfx_i+\bfxi)-F_j(\lam \bfy_i+\bfxi))=0\quad (1\le j\le r),$$
whence the system (\ref{9.1}) is translation-dilation invariant. Such systems are easily 
generated by taking one or more seed polynomials $G(\bfx)$, and then appending to the 
system the successive partial derivatives with respect to each variable. Without loss, one 
may then consider only reduced systems $\bfF$ in which $F_1,\ldots ,F_r$ are linearly 
independent over $\dbQ$.\par

With the above system $\bfF$, we introduce some parameters in order to ease subsequent 
discussion. We refer to the number of variables $d=d(\bfF)$ in $\bfF$ as the 
{\it dimension} of the system, the number of forms $r=r(\bfF)$ comprising $\bfF$ as its 
{\it rank}, and we denote by $k_j=k_j(\bfF)$ the total degree of $F_j$. Finally, the 
{\it degree} $k=k(\bfF)$ of the system is defined by
$$k(\bfF)=\max_{1\le j\le r}k_j(\bfF),$$
and the {\it weight} $K=K(\bfF)$ of the system is
$$K(\bfF)=\sum_{j=1}^rk_j(\bfF).$$
Then by adapting the methods of this paper in a pedestrian manner within the arguments 
of \cite{PPW2013}, one may establish the following.

\begin{theorem}\label{theorem9.1} Let $\bfF$ be a reduced translation-dilation invariant 
system of polynomials having dimension $d$, rank $r$, degree $k$ and weight $K$. 
Suppose that $s$ is a natural number with $s\ge r(k+1)$. Then for each $\eps>0$, and 
any complex sequence $(\gra_\bfn)_{\bfn\in \dbZ^d}$, one has
$$\oint \biggl| \sum_{\bfn\in [-N,N]^d}\gra_\bfn e(\alp_1F_1(\bfn)+\ldots +\alp_rF_r(\bfn))
\biggr|^{2s}\d\bfalp \ll N^{sd-K+\eps}\biggl( \sum_{\bfn\in [-N,N]^d}|\gra_\bfn|^2
\biggr)^s.$$
Moreover, one may take $\eps=0$ when $s>r(k+1)$.
\end{theorem}

The arguments required for the proof of Theorem \ref{theorem9.1} are straightforward 
analogues of those required in the case $d=1$ central to this paper, and involve none of the 
complications demanded by the multigrade efficient congruencing methods of 
\cite{Woo2014a, Woo2014b, Woo2014c}. We consequently propose to expand no further on 
this subject, leaving the reader to complete the routine exercises needed for its proof, 
and to apply \cite{PPW2013} as the necessary framework.

\bibliographystyle{amsbracket}
\providecommand{\bysame}{\leavevmode\hbox to3em{\hrulefill}\thinspace}

\end{document}